\documentclass[letterpaper,10pt]{amsart}

\input xy  
\xyoption{all}

\usepackage{amsmath,amssymb,bm,amsthm, textcomp}
\usepackage{hyperref,comment}
\usepackage{fullpage}
\usepackage{graphicx, psfrag}
\usepackage[usenames]{color}

\theoremstyle{plain}
\newtheorem{thm}{Theorem}[section]
\newtheorem{lemma}[thm]{Lemma}
\newtheorem{prop}[thm]{Proposition}
\newtheorem{cor}[thm]{Corollary}
\newtheorem{question}{Question}

\newtheorem*{thm*}{Theorem}
\newtheorem*{clcsintro}{Theorem \ref{classicalcase}}

\theoremstyle{definition}
\newtheorem*{defn}{Definition}

\newtheorem*{assumption}{Assumption}
\newtheorem{rmk}[thm]{Remark}
\newtheorem{example}[thm]{Example}

\newcommand{\mc}[1]{\mathcal{#1}}
\newcommand{\frk}[1]{\mathfrak{#1}}

\def\C{\mathbf{C}}
\def\Z{\mathbf{Z}}

\def\O{\mathcal{O}}
\def\A{\mathcal{C}}
\def\D{\mathcal{B}}

\def\p{\mathfrak{p}}

\def\bt{\bullet}

\def\ta{\tilde{a}}

\def\br{\mathbf{r}}

\def\tp{t_+}
\def\tm{t_-}

\def\Op{\O^+}

\def\sg{\sigma}
\def\Sg{\Xi}
\def\la{\lambda}

\def\aart{\A^{\gr}_{Art}}

\def\onto{\twoheadrightarrow}

\def\tensor{\otimes}
\def\empt{\varnothing}
\def\iso{\stackrel{\simeq}{\longrightarrow}}

\DeclareMathOperator{\Hom}{Hom}
\DeclareMathOperator{\Ext}{Ext}
\DeclareMathOperator{\End}{End}
\DeclareMathOperator{\gr}{gr}
\DeclareMathOperator{\mult}{mult}
\DeclareMathOperator{\id}{id}
\DeclareMathOperator{\Aut}{Aut}
\DeclareMathOperator{\spec}{Spec}
\DeclareMathOperator{\maxspec}{MaxSpec}

\DeclareMathOperator{\supp}{supp}
\DeclareMathOperator{\ann}{ann}
\DeclareMathOperator{\proj}{Proj}
\DeclareMathOperator{\sym}{Sym}
\DeclareMathOperator{\grMod}{grMod}

\def\Homgr{\Hom_{\gr}}
\def\Extgr{\Ext_{\gr}}
\def\Endgr{\End_{\gr}}

\title{Generalized Weyl algebras: category $\O$ and graded Morita equivalence}
\author{Ian Shipman}
\email{ics@math.uchicago.edu}
\address{Dept. of Mathematics, University of Chicago, 5734 S. University Ave., Chicago, IL 60637}

\begin{document}

\maketitle

\begin{abstract}
We study the structural and homological properties of graded Artinian modules over generalized Weyl algebras (GWAs), and this leads to a decomposition result for the category of graded Artinian modules.  Then we define and examine a category of graded modules analogous to the BGG category $\O$.  We discover a condition on the data defining the GWA that ensures $\O$ has a system of projective generators.  Under this condition, $\O$ has nice representation-theoretic properties.  There is also a decomposition result for $\O$.  Next, we give a necessary condition for there to be a strongly graded Morita equivalence between two GWAs.  We define a new algebra related to GWAs, and use it to produce some strongly graded Morita equivalences.  Finally, we give a complete answer to the strongly graded Morita problem for classical GWAs.
\end{abstract}

\section{Introduction}
Generalized Weyl algebras (GWAs) are defined simply by generators and relations.
\begin{defn}[Generalized Weyl algebra]
Let $R$ be a finitely generated $\C$ algebra, $\sg \in \Aut_\C(R)$ a $\C$ linear automorphism, and $v \in R$ an element.  We refer to such a triple $(R,\sg,v)$ as \emph{GWA data}.  This data determines an algebra $T(R,\sg,v)$ that is generated by $R,\tp,$ and $\tm$ subject to the relations
\begin{equation}\label{gwarelations}
 r \tp = \tp \sg(r), \quad \sg(r) \tm = \tm r, \quad \tm\tp = \sg(v), \quad \tp\tm = v
\end{equation}
for any $r \in R$.  $R$ is naturally a subring of $T(R,\sg,v)$.
\end{defn}
Any GWA has a natural $\Z$ grading and throughout the article we will view GWAs in the graded context.  The grading is determined by assigning the the following degrees to the generators
\[
\deg(\tp) = 1, \quad \deg(\tm) = -1, \quad \deg(r) = 0, \, r \in R.
\]
To deal with graded objects we abide by some standard conventions.  We denote the $i^{th}$ homogeneous component of $M$ by $M_i$, define the $n$-shift $M[n]$ of $M$ by the grading rule $M[n]_i = M_{i+n}$, and we set $\delta(M) = \{i \in \Z : M_i \neq 0\}$.  We use the notation $\Homgr$ for the space of module maps that preserve degree, so a module homomorphism $\phi$ is in $\Homgr(M,M')$ if $\phi(M_i) \subset M'_i$.

In addition to the study of arbitrary GWAs, we will be interested in a certain special case.
\begin{defn}[Classical GWAs]
The GWAs determined by data of the form $(\C[h],\tau,v)$ where $\tau(p(h)) = p(h+1)$ are called \emph{classical GWAs}.  We introduce a simpler notation for the classical GWAs: $T(v) = T(\C[h],\tau,v)$.  From now on $\tau$ will always refer to the above automorphism of $\C[h]$.
\end{defn}

GWAs have been well studied in a series of papers by V. Bavula and others, including \cite{Bav1,Bav2, BJ}.  One motivation for the study of GWAs is the fact that the classical Weyl algebra of differential operators on $\C[h]$ is isomorphic to $T(h)$.  In addition, Hodges studied the classical GWAs in \cite{H} as noncommutative deformations of Kleinian singularities.  There, he posed the Morita problem.  While certain necessary conditions are known, e.g. \cite{RS}, the Morita problem remains open.  Motivated by these results, we introduce a version of the Morita problem that incorporates the natural grading on the GWAs.
\begin{defn}
Let $A$ and $B$ be graded rings. A \emph{strongly graded Morita equivalence} is a $\C$-linear equivalence of categories $F:B-\grMod \iso A-\grMod$ such that for any graded $B$ module $M$, $F(M[1])\cong F(M)[1]$.  
\end{defn}
In \S 4 we discover both a necessary condition for there to be a strongly graded Morita equivalence between two GWAs and a method for constructing such equivalences.  Along the way, we define an algebra that generalizes both the GWAs and the preprojective algebra.  The main result of the paper is the following theorem, where ``type'' refers to a certain explicit equivalence relation on $\C[h]$.
\begin{clcsintro}
$T(v_1)$ and $T(v_2)$ are strongly graded Morita equivalent if and only if for some $b$, $v_1(h+b)$ and $v_2(h)$ have the same type.
\end{clcsintro}

It seems to be well known that the graded simple modules of GWAs can be classified in a straightforward manner if one knows the maximal ideals of $R$ and understands the action of $\sg$ on them.  However, we show that using only the abstract structure of the category of graded modules, one can recover when two simple modules are related by the action of $\sg$.  This is one of the observations at the heart of the theorem above.  In section 2 we will study graded Artinian modules and eventually show that the homological algebra of the category of graded Artinian modules reveals information about the action of $\sg$.

In section 3 we define a certain category of modules, $\Op$, and we are not aware of any other treatment of this category in the literature in the case of GWAs.  The idea behind the category is well known and goes back to \cite{BGG}.  In fact, Khare \cite{K} gives a very general treatment of a category analogous to $\Op$ when the algebra of study has a ``triangular'' decomposition.  The GWAs are not triangular, however they are quotients of triangular algebras.  Although we do not spell out a comparison in this article, our $\Op$ and the category $\O$ defined by Khare have many properties in common, at least in the classical case.  It turns out to be highly structured and preserved under strongly graded Morita equivalences.  We understand the strongly graded Morita equivalences in part by studying how they mutate the structure of $\Op$.  $\Op$ has analogues of the familiar representation theoretic properties of the BGG category $\O$.  We introduce a condition ($\ast$) that holds trivially in the classical case, under which $\Op$ is well behaved.  For example, when ($\ast$) holds, $\Op$ has enough projectives and when ($\ast$) fails, $\Op$ may not have enough projectives.  We use the geometry of the zero set of $v$ to give a decomposition of $\Op$.  Also, this subcategory ``localizes'' on the zero set of $v$, at least when ($\ast$) holds, since it is equivalent to the module category of a finite (but not necessarily commutative) $R/v$ algebra.

Finally, in section 4 we study strongly graded Morita equivalence in earnest.  We construct a map from $R$ into the center of the category of graded modules over a GWA and it turns out that any strongly graded Morita equivalence induces a map between the centers of the graded module categories which is compatible with these embeddings.  In fact, under a mild condition what we prove implies that if $T(R_1,\sg_1,v_1)$ and $T(R_2,\sg_2,v_2)$ are strongly graded Morita equivalent, then not only are $R_1$ and $R_2$ isomorphic, but the zero set of $v_1$ has a locally closed partition such that translating the parts by iterates of $\sg$ gives the zero set of $v_2$.  The notion of ``type'', introduced in \S 4, illustrates this in the case when the zero sets of $v_1$ and $v_2$ are collections of points.  

We are also able to describe a method for producing many strongly graded Morita equivalences between GWAs, which leads to a generalization of a sufficient condition for Morita equivalence discovered by Hodges.  The main tool is a ``many-vertex'' version of a GWA, which has the property that attached to each vertex is an ordinary GWA and there is a simple criterion for when two of the vertex GWAs are strongly graded Morita equivalent.

\subsubsection*{Acknowledgements}
I would like to thank Victor Ginzburg for introducing me to category $\O$ for generalized Weyl algebras, for suggesting that it might be interesting, and for many helpful discussions.  Also, I thank Aaron Marcus and Mike Miller for reading a draft of this article, catching mistakes, and offering helpful comments.  

\section{Artinian graded modules}

In this section, we will study the structure and certain homological properties of Artinian graded modules over GWAs.  Artinian modules are examples of what are known as weight modules and one can find a general treatment with constructions and structure theory in \cite{DGO}.  However for the sake of completeness, we will reproduce the results we need.   The connection between simple graded modules and $\maxspec R$; and the homological properties that we develop will be used in section 4 to prove the main results.  Lemma \ref{simp} is used frequently throughout the paper.  We will always use the term ``map'' to mean homomorphism.  As a final preliminary remark, every module will be graded, every sub- and quotient module will be graded, and every map will be degree preserving.

Let $R$ be a finitely generated commutative $\C$ algebra.  Fix a $\sg \in \Aut(R)$ and an element $v \in R$, and let $A := T(R,\sg,v)$.  Let $Z = \spec R/(v) =  \{ \p \in \spec R : v \in \p \}$ and $Z^\sg = \{\sg^n(\p) : \p \in Z, n \in \Z\}$.  We will think of these sets as spaces with their Zariski (subspace) topology but will not need to think of $Z$ as a scheme, and $Z^\sg$ might not even have an evident scheme structure.  It will be helpful to think of $A_i$ as an $(R,R)$ bimodule generated by $\tp^i$ or $\tm^i$.  Note that $A_i$ is isomorphic to $R$ as both a left and right module.  From now on, we reserve the notations $\la$ and $\mu$ for \emph{maximal} ideals of $R$.  

\begin{defn}
For each $\la \in \maxspec R$ we can view $\la$ as a subset $\la \subset A_0$ and define a graded left $A$-module $A^\la = A/A\la$.   These modules will play a very prominent role.  Note that $A \la = \la \oplus \bigoplus_{n > 0}{ \tp^n \la \oplus \tm^n \la }$ and therefore for all $i \in \Z$ we have $\dim_\C A^\la_i = 1$.  Since the images of $t^n_\pm$ are nonzero for any $n \geq 0$, the quotient map gives the identification 
\[
A \supset \bigoplus_{n > 0}{\C \tm^n} \oplus \C \oplus \bigoplus_{n > 0}{\C \tp^n} \stackrel{\simeq}{\longrightarrow} A^\la
\]
of $\C$ vector spaces.  As a left $R$ module we have
\[
	A^\la \cong \bigoplus_{n > 0}{ (R/\sigma^{n}(\la)) \cdot \tm^n } \oplus R/\la \oplus \bigoplus_{n>0}{(R/\sigma^{-n}(\la)) \cdot \tp^n}
\]
Set $\chi_\la = \{k \in \Z : \sg^k(v) \in \la\}$.  For $k \in \chi_\la$ set
\[
A^{\la,k} :=
\begin{cases}
\bigoplus_{i < k}{A^\la_i} & k \leq 0, \\
\bigoplus_{i \geq k}{A^\la_i} & k > 0.
\end{cases}
\]
For any $k \in \chi_\la$, $A^{\la,k}$ is a submodule of $A_\la$.  Clearly, $\tm A^\la_i \subset A^\la_{i-1}$ and $\tp A^\la_i \subset A^\la_{i+1}$.  For $i \leq 0$, $\tm A^\la_i = A^\la_{i-1}$ and for $i \geq 0$, $\tp A^\la_i = A^\la_{i+1}$.  Therefore $A^{\la,k}$ is a submodule if $\tm A^\la_k = 0$ when $k > 0$ or $\tp A^\la_{k-1} = 0$ when $k \leq 0$.  In the first case we have $\tm \tp^k = \tp^{k-1}\sg^k(v) = 0$ and in the second case $\tp \tm^{1+|k|} = \tm^{|k|}\sg^k(v) = 0$ since $k \in \chi_\la$.
\end{defn}

Right multiplication by $\tp$ and $\tm$ defines left module maps $\tilde{\phi}^+:A \to A[1]$ and $\tilde{\phi}^-:A[1] \to A$, respectively.  Observe that $\tilde{\phi}^+(A\la) = A\la \tp = A\tp \sg(\la) \subset A \sg(\la)$ and similarly $\tilde{\phi}^-(A \sg(\la)) \subset A \la$.  Thus right multiplication by $\tp$ and $\tm$ induce a pair of maps $\phi^+:A^\la \to A^{\sg(\la)}[1]$ and $\phi^-:A^{\sg(\la)}[1] \to A^\la$ given by
\[
\phi^+(x) = x\tp, \quad \phi^-(x) = x \tm.
\]
Moreover for homogeneous $x$, $\phi^+ \circ \phi^-(x) = \sg^{1-\deg(x)}(v)x$ and $\phi^- \circ \phi^+(x) = \sg^{-\deg(x)}(v) x$.  Finally, recall that if $M$ is a graded $A$ module define $\delta(M) := \{i \in \Z : M_i \neq 0\}$.  The following Lemma is a version of Theorem 5.8 in \cite{DGO}.

\begin{lemma} \label{simp} \mbox{}
\begin{enumerate}  
\item  Let $M \subset A^\la$ be a proper, nontrivial, graded submodule.  Then either $M = A^{\la,k}$ for some $k \in \chi_\la$ or $M = A^{\la,k} \oplus A^{\la,k'}$ for $k, k' \in \chi_\la$ with $k \leq 0 < k'$.
\item $A_\la$ is Artinian if and only if $\chi_\la$ is finite.  It is simple if and only if $\chi_\la = \empt$, or put another way $\la \notin Z^\sg$.
\item There is a unique maximal submodule of $A^\la$.  Let $S^\la$ be the quotient of $A^\la$ by this maximal submodule.  Note that $S^\la$ is simple.  Every simple graded $A$ module is isomorphic to $S^\la[i]$ for some $\la \in \maxspec R$ and some $i \in \Z$.
\item If $\la \notin Z$ then $A^\la \cong A^{\sg(\la)}[1]$.
\end{enumerate}
\end{lemma}
\begin{proof}
(i)  Let $M \subset A^\la$ be a proper, nontrivial, graded submodule.  Since the homogeneous components of $A^\la$ are one dimensional the set $\delta(M)$ determines $M$.  Since $A^\la$ is cyclic and hence generated by $A^\la_0$, $0 \notin \delta(M)$.  Let $i \in \delta(M)$.  If $i > 0$ and $j > i$ then $j \in \delta(M)$ since $\tp^{j-i}A^\la_i = A^\la_j$.  Similarly if $i < 0$ and $j < i$ then $j \in \delta(M)$.  Hence $\delta(M)$ is determined by $k_1 = \max\{k \in \delta(M) : k < 0\}$ and $k_2 = \min\{k \in \delta(M) : k > 0\}$.  Suppose that $k_1$ exists.  Since $k_1 + 1 \notin \delta(M)$ we must have $\tp A^\la_{k_1} = 0$ and therefore $\tp \tm^{-k_1} = \tm^{-k_1 - 1} \sg^{k_1 + 1}(v) = 0$ so $\sg^{k_1+1}(v) \in \la$.  But this means that $k_1 + 1 \in \chi_\la$ so $A^{\la,k_1+1} \subset M$.  Similarly if $k_2$ exists then $A^{\la,k_2} \subset M$.  Thus $M = A^{\la,k_1+1}, M = A^{\la,k_2},$ or $M = A^{\la,k_1+1} \oplus A^{\la,k_2}$, depending on whether $k_1$ or $k_2$ or both exist.

(ii)  Immediate from (i).

(iii) The first assertion is clear from (i).  Suppose $S$ is a simple graded $A$ module and that $i \in \delta(S)$.  Then there is a surjection $A \onto S[i]$.  Let $J$ be the kernel of this surjection and let $\la \in \maxspec R$ be a maximal ideal containing $J_0$.  Consider $J + A\la$.  Since $S[i]$ is simple the image of this left ideal must be either $0$ or $S[i]$.  Because $(J+A \la)_0 = \la$ it follows that $J + A \la \subset J$ and therefore that $A \la \subset J$.  Hence, our surjection factors through a map $A^\la \onto S[i]$.  Since $S^\la$ is the only simple quotient of $A^\la$ it follows that the previous map factors through $S^\la \to S[i]$.  Of course, $S^\la[-i]$ and $S$ are both simple so the map $S^\la \to S[i]$ has to be an isomorphism.  We conclude that $S \cong S^\la[-i]$.

(iv)  Consider the maps $\phi^+:A_\la \to A_{\sg(\la)}[1]$ and $\phi^-:A_{\sg(\la)}[1] \to A_\la$ from above.  We have $\phi^+ \circ \phi^-(x) = x \sg(v)$ and $\phi^- \circ \phi^+(x) = x v$.  Since $v \neq 0$ modulo $\la$ or equivalently $\sg(v) \neq 0$ modulo $\sg(\la)$ we have $x \sg(v) \neq 0$ and $x v \neq 0$.  But this means that $\phi^+ \circ \phi^-$ and $\phi^- \circ \phi^+$ coincide with multiplication by a nonzero element of $\C$.  So $\phi^+$ and $\phi^-$ are isomorphisms.
\end{proof}
Let $\la \in \maxspec R$.  We want to locate all of the simple subquotients of $A^\la$.  First, if $\chi_\la = \empt$ then $A^\la$ is already simple.  So assume that $\chi_\la \neq \empt$.  Let $A^{\la,+}$ and $A^{\la,-}$ be the maximal positively and negatively graded submodules, respectively.  Lemma \ref{simp} implies that $A/(A^{\la,+} \oplus A^{\la,-})$ is simple and that the submodules of $A^{\la,+}$ and $A^{\la,-}$ form decreasing filtrations.  So every simple subquotient of $A^\la$ except $S^\la$ is either of the form $A^{\la,k}/A^{\la,l}$ for $l > k > 0$ consecutive elements of $\chi_\la$, or of the form $A^{\la,l}/A^{\la,k}$ for $k < l \leq 0$ consecutive elements of $\chi_\la$.  From this we see that if $S$ and $T$ are distinct simple subquotients of $A^\la$ then $\delta(S) \cap \delta(T) = \empt$.  Hence each homogeneous component of $\bar{A}^\la := \bigoplus S$, where the sum runs over the simple subquotients of $A^\la$, is one dimensional, and $\bar{A}^\la = A^\la$ unless $\la \in Z^\sg$.

Now, let us examine part (iv) of Lemma \ref{simp} and its proof more closely.  Suppose that $\la \in Z$.  This means that $v \in \la$ and so $0 \in \chi_\la$.  Hence $A^\la$ and $A^{\sg(\la)}[1]$ have special submodules $A^{\la,0}$ and $A^{\sg(\la),1}[1]$.  We see from the proof that $\phi^+ \circ \phi^-$ and $\phi^- \circ \phi^+$ are both zero in this situation.  However, for $i \geq 0$ we have $\phi^+( A^\la_i) = A^{\sg(\la)}_{i+1}$ and for $i \leq 0$ we have $\phi^-(A^{\sg(\la)}_i) = A^{\la}_{i-1}$.  It follows that $\phi^+$ induces an isomorphism $A^\la/A^{\la,0} \iso A^{\sg(\la),1}[1]$ and that $\phi^-$ induces an isomorphism $A^{\sg(\la)}[1]/A^{\sg(\la),1}[1] \iso A^{\la,0}$.  Hence 
\[ A^{\la,0} \oplus A^\la/A^{\la,0} \cong \left( \left( A^{\sg(\la)}/A^{\sg(\la),1} \right) \oplus A^{\sg(\la),1}\right)[1]. \]  

\begin{lemma} \label{alamsubq} \mbox{}
 \begin{enumerate}
  	
	\item $\bar{A}^\la \cong \bar{A}^{\sg(\la)}[1]$.  Therefore if $\la = \sg^n(\la)$ then $\bar{A}^\la \cong \bar{A}^\la[n]$.
	\item There exist simple subquotients $S,T$ of $A^\la$ such that $S \cong T[n]$ if and only if $\la = \sg^n(\la)$.  In particular if $\sg$ acts freely on $\maxspec R$ then the simple subquotients of $A^\la$ are distinct.
 \end{enumerate}
\end{lemma}
\begin{proof}
(i)  First, if $\la \notin Z$ then $A^\la \cong A^{\sg(\la)}[1]$ and therefore $\bar{A}^\la \cong \bar{A}^{\sg(\la}[1]$.  Now suppose $\la \in Z$.  For any $A$ module $M$ we can form $\bar{M} = \bigoplus S$ where the sum is over simple subquotients of $M$.  Of course, if $M'' = M/M'$ then $\bar{M} = \bar{M'} \oplus \bar{M''}$.  In the previous paragraph we saw that
\[ A^{\la,0} \oplus A^\la/A^{\la,0} \cong \left( \left( A^{\sg(\la)}/A^{\sg(\la),1} \right) \oplus A^{\sg(\la),1}\right)[1] \]
and therefore $\bar{A}^{\la} \cong \bar{A}^{\sg(\la)}[1]$. 

(ii)  Suppose that $S \to T[n]$ is an isomorphism and that $S_i \neq 0$.  Observe that as left $R$ modules, $S_i \cong A^{\la}_i \cong R/\sg^{-i}(\la)$ and $T_{i+n} \cong A^\la_{i+n} \cong R/\sg^{-i-n}(\la)$.  Since $R/\sg^{-i}(\la) = S_i \to T_{i+n} = R/\sg^{-i-n}(\la)$ is an isomorphism of $R$ modules we see that $\la = \sg^n(\la)$.  The converse follows from part (i).
\end{proof}

\begin{rmk}
Write $\bar{A}^\la = \dotsb \oplus S_{-1} \oplus S^\la \oplus S_1 \oplus \dotsb$, where each $S_j$ is simple.  Let $k = \min \delta(S_1)$.  By \ref{alamsubq}, $\bar{A}^\la \cong \bar{A}^{\sg^{-k}(\la)}[-k]$.  Therefore $S_1 \cong S^{\sg^{-k}(\la)}[-k]$.  Similarly, if $k = \max \delta(S_{-1})$ then $S_{-1} \cong A^{\sg^{-k}(\la)}[-k]$.  To compute the relevant integers we note that $A^\la$ has a unique maximal submodule which splits into a possibly trivial direct sum $M^- \oplus M^+$ where $\delta(M^-), \delta(M^+)$ consist of negative and positive integers respectively.  Each of $M^-,M^+$ has a decreasing filtration $F^iM^-, F^i M^+$ with simple quotients.  Now $F^iM^-/F^{i+1}M^- = S_{-i}$ and $F^iM^+/F^{i+1}M^+ = S_i$.  By \ref{simp}, if we enumerate $\chi_\la = \{ \dotsb < k_{-2} < k_{-1} \leq 0 < k_1 < k_2 < \dotsb \}$ then $\max \delta(F^i M^-) = k_{-i}-1$ and $\min \delta(F^i M^+) = k_i$.  Therefore 
\[
\bar{A}^\la = \left( \bigoplus_{k \in \chi_\la, k \leq 0}{S^{\sg^{1+|k|}(\la)}[1+|k|]} \right) \oplus S^\la \oplus \left( \bigoplus_{k \in \chi_\la, k > 0}{S^{\sg^{-k}(\la)}[-k]}\right).
\]
\end{rmk}

We can view $A[n]$ as an $(A,R)$ bimodule as follows.  The left action of $A$ is just the usual left action.  Let $x \in A[n]$ and let $r \in R$ then $x \cdot r = x \sg^{n}(r)$ where the undotted action is just multiplication.  We can identify $\Homgr(A[n],A[m]) \cong A_{m-n}$, where $a \in A_{m-n}$ corresponds to the map $x \mapsto xa$.  Therefore, to check that the maps $\Homgr(A[n],A[m])$ are compatible with the right $R$ module structure, we only need to check that $\phi^+$ and $\phi^-$, corresponding to right multiplication by $\tp$ and $\tm$ respect this structure.  For $x \in A$ and $r \in R$ we have $\phi^+( x \cdot r ) = \phi^+(xr) = xr\tp = x\tp \sg(r) = \phi^+(x) \cdot r$.   A similar check verifies that $\phi^-$ respects the right $R$ module structure.  We can formulate the relations \eqref{gwarelations} as 
\begin{equation}\label{phirelations}
\phi^+\circ \phi^-(x) = x \cdot v, \, x \in A[1], \quad \phi^- \circ \phi^+(x) = x \cdot v, \, x \in A.
\end{equation}
Now, because $A[n]$ are projective generators for the category of graded $A$ modules, any graded $A$ module is naturally an $(A,R)$ bimodule.  

If $x_1,\dotsc,x_k$ are homogeneous generators of a graded $A$ module $M$ then the set $\{\tp^{|i - \deg x_j|}x_j, \tm^{|\deg x_j-i|}x_j\}$ generates $M_i$ as both a left and right $R$ module.  Therefore each $M_i$ is a finitely generated left and right $R$ module.  Let $M$ be a graded $A$ module equipped with this bimodule structure.  Recall that if $\bar{M}$ is a finitely generated $R$ module then we have the support $\supp(\bar{M}) = \{\frk{p} \in \spec R : \ann(\bar{M}) \subset \frk{p}\}$.  For $x \in M_i$ and $r \in R$ we have $r x = x \sg^i(r)$.  So the support of $M_i$ as a left module differs from the support of $M_i$ as a right module by the action of $\sigma^i$ on $\spec R$. This means that many of the properties of the support, such as dimension do not depend on whether we view $M$ as a left or right module.  We will use this bimodule structure as a matter of course in \S\S 3, 4.  We can think of this natural right module structure as giving a map of $\C$ algebras from $R$ to the center of the category $A-\grMod$.

\begin{lemma} \label{artgeom}
If $M$ is a Artinian graded $A$ module then for every $i$ the support of $M_i$ in $\spec R$ is finite.  If $\chi_\la$ is finite for every $\la \in \maxspec R$ then the converse is true.
\end{lemma}
\begin{proof}
Assume that $M$ is a Artinian graded $A$ module.  Then for any ideal $J \subset R$ and $i \in \Z$, the chain of modules $M \supset M J \supset \dotsb M J^k \supset \dotsb$ has to stabilize.  Since $M$ is Artinian, $M$ is finitely generated so $M_i$ is finitely generated.  Specializing to $J=\lambda \in \maxspec R$ we see that $M \lambda^n = M \lambda^{n+1}$ for some $n$.  Either $M \la = M$ or else for every $i$, $\lambda^n M_i = 0$ by Nakayama's lemma.  This means that every maximal ideal of $R/\ann(M)$ is nilpotent and we conclude that $M_i$ has to have finite support.

Now we prove the converse under the additional assumption that $\chi_\la$ is finite for every $\la \in \maxspec R$.  Since $M$ is generated by finitely many cyclic modules it is enough to show that a cyclic module whose components have finite support is Artinian.  So consider $A/J$ where $J$ is a homogeneous left ideal.  Note that $J_0$ contains an ideal of the form $\prod \la_i^{e_i}$ for some $\la_i \in \maxspec R$.  We replace $J$ by $A \prod \la_i^{e_i}$ so that $A/J = \oplus_i A/A \la_i^{e_i}$.  $A/A \la_i^{e_i}$ clearly has a filtration by modules such that the quotients are $A_{\la_i}$.  By \ref{simp}, $A_{\la_i}$ is Artinian if and only if $\chi_{\la_i}$ is finite.
\end{proof}

\begin{cor}
Suppose that $\chi_\la$ is finite for every $\la \in \maxspec R$.  Let $J$ be a homogeneous left ideal of $A$ such that $A/J$ is Artinian.  Then $A/(A \cdot J_0)$ is Artinian.
\end{cor}
\begin{proof}
By \ref{artgeom} we know that $A_0/J_0$ has finite support.  Also by \ref{artgeom}, it is sufficient to check that $(A/A J_0)_i$ has finite support for all $i$.  Finally note that $(A/A J_0)_i \cong (A/A J_0)_{i+1}$ as a right $R$ module.
\end{proof}
\begin{cor}
 If $M$ is a finitely generated, Artinian, graded $A$ module then there is an $N = N(M)$ such that $\dim_\C M_n \leq N$ for all $n$.  
\end{cor}
\begin{proof}
 It is sufficient to show that this is true for a Artinian graded quotient of $A$ itself.  Let $J$ be a homogeneous left ideal such that $A/J$ is Artinian.  By \ref{artgeom}, $J$ contains $J' = A \prod \la_i^{e_i}$ for some $\la_i \in \maxspec R$ and thus $A/J' = \bigoplus_i A/A\la_i^{e_i}$ maps onto $A/J$.  Finally note that $(A/A\la_i^{e_i})_j$ is isomorphic to $R/\la_i^{e_i}$ as a right $R$ module and therefore $\dim_\C (A/J)_j \leq \dim_\C (A/A\la_i^{e_i})_j = \dim_\C R/\la_i^{e_i} < \infty$.  
\end{proof}

We now describe a \emph{duality functor} on $\aart$.  Let $\iota:A \to A^{op}$ be the anti-involution of $A$ which is the identity on $R$ and satisfies $\iota(\tp) = \tm, \iota(\tm) = \tp$.  Note that $\iota$ reflects the grading in the sense that $\iota(A_n) = A_{-n}$.  We define the duality functor taking $M \mapsto M^*$ where $M^*_n = \Hom_\C(M_n, \C)$ and for $x \in A_i$, $\phi \in M^*_j$ and $m \in M_{i+j}$ we have $(x\phi)(m) = \phi(\iota(x)m)$.  The right module structure is the obvious one, $(\phi a)(m) = \phi(ma)$.  Observe that $(M[n])^* = M^*[n]$ and that $M^*$ is simple if and only if $M$ is simple.  Of course, $M^*$ is defined for any graded $A$ module $M$, but $M^{**}$ will not be isomorphic to $M$ if $M$ does not have finite dimensional homogeneous components.

Let $\aart$ be the category of Artinian, graded $A$ modules.  As an application of the last two lemmas and a few more, we will show how to decompose $\aart$ using points of $\maxspec R$ and the $\sg$ action.  Say that two simple subquotients $S,T$ of $A^\la$ are \emph{adjacent} if $\delta(S \oplus T)$ is an interval.  An interval is a subset of $\Z$ of the form $\{k \in \Z : n \leq k \leq N \}$ where we allow $n = -\infty$ and $N = \infty$.  Recall that if $S$ and $T$ are distinct simple subquotients then $\delta(S) \cap \delta(T) = \empt$.  Therefore we can define a total order on the simple subquotients by setting $S < T$ if $i < j$ for any $i \in \delta(S)$ and $j \in \delta(T)$.  Intuitively, $S$ sits to the left of $T$.  Then the adjacent modules are those which are adjacent with respect to this ordering.  We say that an arbitrary pair of simple $A$ modules $S,T$ is adjacent if they are adjacent simple subquotients for some $A^\la$.

In order to do some homological algebra, let us fix notation.  For graded left $A$ modules $M,N$ let $\Homgr(M,N)$ be the space of degree preserving module maps.  Let $\Extgr^*$ be the derived functor of $\Homgr$ in the category of all graded modules.  As in the ungraded situation, $\Extgr^p(M,N)$ is the space of equivalence classes of extensions of length $p$ of $M$ by $N$.  Observe that for $M',M''$ Artinian graded modules if $0 \to M \to M \to M'' \to 0$ is a short exact sequence then $\dim_\C M_i < \infty$ for each $i$.  Combining this with the fact that $(-)^*$ is exact we see that $(-)^*$ gives an isomorphism $\Extgr^1(M'',M') \cong \Extgr^1(M'^*,M''^*)$.

\begin{lemma}\label{simpleselfduality} 
If $S$ is a simple module then $S \cong S^*$.
\end{lemma}
\begin{proof}
By Lemma \ref{simp}, every simple module is a shift of a module of the form $S^\la$.  Therefore it suffices to check that $(S^\la)^* \cong S^\la$.  Let $0 \neq \phi \in (S^\la)^*_0$ and let $A \onto (S^\la)^*$ be the homomorphism defined by $a \mapsto a \phi$.  Then for $m \in (S_\la)_0$ we have $(a\phi)(m) = \phi(am) = 0$ for all $a \in \la$.  Hence $(S^\la)^*$ is a simple quotient of $A^\la$ and must be isomorphic to $S^\la$.
\end{proof}

\begin{lemma}\label{nontrivialext}
Given two adjacent simple subquotients $S,T$ of $A^\la$ there is a nonzero extension class in $\Extgr^1(S,T)$.
\end{lemma}
\begin{proof}
Every simple subquotient of $A^\la$ other than the simple quotient is a subquotient of $A^{\lambda,k}$ for some $k \in \chi_\la$.  Since $\Extgr^1(S,T) = \Extgr^1(T^*,S^*) = \Extgr^1(T,S)$, we can interchange $T$ and $S$ if we want.  We suppose that $S < T$.  We deal with two cases.  Say $T < S^\lambda$.  Then there are $A^{\la,k}$ and $A^{\la,l}$ such that we have a short exact sequence $0 \to S \to A^{\la,k}/A^{\la,l} \to T \to 0$.  This cannot be split because there are no incomparable submodules of $A^{\la,k}$, or in other words Lemma \ref{simp} implies that if $M,M' \subset A^{\la,k}$ then $M \subset M'$ or $M' \subset M$.  Thus the exact sequence determined by the module is nontrivial in $\Extgr^1(T,S)$.  If $S^\la < S$ then a similar construction gives a nonzero extension class in $\Extgr^1(S,T)$.  Suppose that $S = S^\lambda$.  Then there is a submodule $M$ and a short exact sequence $0 \to T \to A^\la/M \to S \to 0$.  If this were split then $A^\la = M_1 + M_2$ for two proper submodules.  Since $A^\la$ has a unique maximal submodule, this cannot occur.  Thus, our exact sequence defines a nonzero class in $\Extgr^1(S,T)$.  We find ourselves in similar circumstances when $T = S^\lambda$, and obtain a nonzero class in $\Extgr^1(T,S)$.
\end{proof}

\begin{thm} \label{extcomp} Let $\la,\mu \in \maxspec R$.  If $\mu \neq \sg^n(\la)$ then $\Extgr^1(S^\la,S^\mu[n]) = 0$.  If $\mu = \sg^n(\la)$ and $\la \notin Z^\sg$ then $\Extgr^*(S^\la,S^\mu[n]) \cong \Ext_R^*(R/\la,R/\la)$.  If $\la \in Z^\sg$ and we set $\mu = \sg^n(\la)$, then $\Extgr^1(S^\la,S^\mu[n]) \neq 0$ only if $S^\la$ and $S^\mu[n]$ are either adjacent or isomorphic.
\end{thm}
\begin{proof}
Let $\la,\mu \in \maxspec R$ and assume that $\mu \neq \sg^n(\la)$.  Consider an extension 
\begin{equation}\label{extension}
 0 \to S^\mu[n] \to M \to S^\la \to 0.
\end{equation}
Note that $M$ must have finite support.  Hence, there exists an ideal $J \subset R$ such that $R/J$ is Artinian and $M,S^\la,$ and $S^\mu[n]$ are all $R/J$ modules.  There are distinct maximal ideals $\la = \la_0, \la_1, \dotsc, \la_N = \mu$ and positive integers $e_j$ such that $R/J = \oplus_{j=1}^N {R/\la_j^{e_j}}$.  Let $\pi \in R/J$ be the idempotent corresponding to $1 \in R/\la^{e_0}$.  Then we have two exact sequences
\begin{gather*}
0 \to S^\mu[n] \cdot \pi \to M \cdot \pi \to S^\la \cdot \pi \to 0 \\
0 \to S^\mu[n] \cdot (1-\pi) \to M \cdot (1-\pi) \to S^\la \cdot (1-\pi) \to 0
\end{gather*}
whose sum is \eqref{extension}. Since $S^\la \cdot (1-\pi) = 0$ and $S^\mu[n]\cdot \pi = 0$, this means that \eqref{extension} is split.  Now, \eqref{extension} was arbitrary so we see that $\Extgr^1(S^\la,S^\mu[n]) = 0$.

Suppose that $\la \notin Z^\sg$ and let $\mu = \sg^n(\la)$.  Let $F^\bullet \onto R/\la$ be a free resolution of $R$ modules.  Then since $A$ is $R$-free on the right, $A \tensor_R F^\bullet \onto A^\la = S^\la$ is a free resolution.  Since $\mu = \sg^n(\la)$ we know that $n \in \delta(S^\mu)$ and therefore $\Homgr(A,S^\mu[n]) = \Hom_R(R,R/\sg^{-n}(\mu))$  and it follows that $H^*( \Homgr(A \tensor_R F^\bullet, S^\mu[n])) = \Ext^*_R(R/\la,R/\sg^{-n}(\mu)) = \Ext^*_R(R/\la,R/\la)$.  

Finally, let $S$ and $T$ be simple.  Suppose that the extension $0 \to T \to M \to S \to 0$ is not split.  We can shift $S$ until $S \cong S^\la$ for some $\la$.  The map $A \onto S$ lifts to a map $A \to M$.  This map must be surjective since otherwise our extension would be split.  Therefore $M = A/J$ for some homogeneous left ideal $J \subset A$.  If $\dim_\C M_0 = 1$ then $M$ is in fact a quotient of $A_\la$ and $S$ and $T$ are adjacent.  Otherwise $\dim_\C M_0 = 2$ and $M$ is a quotient of $A/A J_0 $ where $J_0 \subset \la$ has codimension 1.  If $\sqrt{J_0} \neq \la$ then $R/J_0 = R/\la \oplus R/\la'$ and $M = S\oplus T$.  Otherwise $A\la/A J_0 = A \tensor_R \la/J_0 \cong A_\la$ because $\la/J_0 \cong R/\la$.  This means that there is a map $A^\la \onto T$ and since $T$ is simple $T \cong S^\la$ and $S \cong T$.
\end{proof}

\begin{rmk}\label{notinzsg} One consequence of this Theorem is that it is possible to tell from the homological algebra alone whether or not a simple graded module is a subquotient of $A^\la$ for some $\la \in Z^\sg$.  Recall that if $\la \notin Z^\sg$ then $S^\la \cong S^{\sg^n(\la)}[n]$.  By Theorem \ref{extcomp} this means that whenever $T$ is a simple graded module such that $\Extgr^1(S^\la[n],T) \neq 0$ we have $T \cong S^\la[n]$.  On the other hand, by Lemmas \ref{nontrivialext} and \ref{simp} if $\la \in Z^\sg$ then there is a simple graded module $T \ncong S^\la[n]$ (but adjacent to $S^\la[n]$) such that $\Extgr^1(S^\la[n],T) \neq 0$.  So a simple graded module $S$ corresponds to $\la \notin Z^\sg$ if and only if whenever $\Ext^1_{\gr}(S,T) \neq 0$ for a simple module $T$ we have $S \cong T$.  The following is a consequence of Lemma \ref{alamsubq}, Lemma \ref{nontrivialext}, and Theorem \ref{extcomp}.
\end{rmk}

\begin{prop}
For $\la \in \maxspec R$ and $n \in \Z$ let $\A_{\la,n}$ be the full subcategory of Artinian graded modules whose simple subquotients are among the simple subquotients of $A^\la[n]$, i.e. summands of $\bar{A}^\la[n]$.  We have $\A_{\la,n} = \A_{\mu,m}$ if and only if $\mu = \sg^{m-n}(\la)$.  Any Artinian graded module $M$ can be decomposed as a direct sum $\bigoplus M_{\la,n}$ with $M_{\la,n} \in \A_{\la,n}$.  Let $\la \notin Z^\sg$.  If $\lambda$ is a regular point of $\maxspec R$ then the homological dimension of $\A_{\la,n}$ is finite and $\dim \A_{\la,n} \leq \dim_\la \spec R$, the Krull dimension of the local ring $R_\la$.
\end{prop}
\begin{proof}
The first and last assertions are immediate from Lemma \ref{alamsubq} and Theorem \ref{extcomp}, respectively.  The second assertion follows from the observation if $\mu \neq \sg^{m-n}(\la)$ and $M \in \A_{\la,n}$ and $N \in \A_{\mu,m}$ then $\supp(M) = \{\sg^{-n}(\la)\}$ and $\supp(N) = \{\sg^{-m}(\mu)\}$ so $\supp(M) \cap \supp(N) = \empt$.
\end{proof}

\section[Category O]{Category $\O$}

We introduce a category $\Op$, the category of graded $A$ modules $M$ with the property that for every $m \in M$, $\tm^n m = 0$ for $n \gg 0$.  Such a module is called \emph{locally nilpotent}.  We define $\O$ to be the category of ungraded  locally nilpotent $A$ modules.  A finitely generated graded $A$ module $M$ belongs to $\O^+$ if and only if $-\infty < \inf \delta(M)$.  Our category $\Op$ behaves very similarly to the familar one associated to a semisimple Lie algebra if $\spec R$ is one dimensional.  See \cite{BGG}.  We reformulate some of the usual properties of $\Op$ so that they carry over to the situation when $\dim \spec R > 1$.  The category $\Op$ is our main object of interest and we will only mention $\O$ a few times, in order to compare it to $\Op$.  One remarkable property of $\O$ is that if our condition ($\ast$) below holds then instead of being Artinian, $\O$ is ``graded'' in the sense that it has a system of projective generators which are graded and such that every homomorphism between them is necessarily degree preserving.  Also, Theorem \ref{gr} gives a block decomposition for $\Op$ parameterized by pairs of a connected component of $Z$ and an integer.  Finally, we will need the fact that $A$ is noetherian, see \cite{Bav1}.

The idea of support and the induced action of $\sg$ on $\spec R$ will play a major role below.  Recall that $\sg$ induces an action $\Sg$ on $\spec R$ by $\Sg( \p ) = \sg^{-1}(\p)$.  For a graded $A$ module $M$, we think of $M$ as an $(A,R)$ bimodule. Let $\ann(M)$ denote the annihilator of $M$ as a \emph{right} $R$ module and $\supp(M) =  \{ \p \in \spec R : \ann(M) \subset \p\}$ is the support as a right module.  Since shifting changes the right module structure, $\ann(M[n]) = \sg^{-n}(\ann(M))$ and $\supp(M[n]) = \Sg^n(\supp(M))$. For a subset $X \subset R$ let $Z(U) = \{ \p \in \spec R : X \subset \p\}$ be the associated closed subset of $\spec R$, the zero locus of $X$.  For $\p \in \spec R$, $Z(\p)$ is the closure of the point $\p$ and is an irreducible subset of $\spec R$.

Let $M$ be a graded $A$ module.  If $\supp(M)$ is disconnected then the right $R$ module structure on $M$ factors through a quotient ring of the form $\prod_W{R_W}$ where $W$ runs over the set of connected components of $\supp(M)$.  We use an orthogonal system of idempotents $1_W = (0,\dotsc,1,\dotsc,0)$ in the quotient ring to define submodules $M_W = M\cdot 1_W$ such that $M = \bigoplus_W{ M_W }$.  From now on we say that $M_W$ is the \emph{summand of $M$ supported on $W$}.  

There are Verma modules in this situation, however they sit very deep inside $\Op$.  Let $\la \in \maxspec R$ and recall that $A^{\la,-1} :=\bigoplus_{i < 0} A^\la_i$ is a graded submodule of $A^\la$ if and only if $\la \in Z$.  In this case, define $V^\la := A^\la/A^{\la,-1}$ viewed as a graded module concentrated in non-negative degrees.  The Verma modules fit into a larger class of modules called big Verma modules that are defined in a similar way, except that instead of corresponding to maximal ideals, the big Verma modules correspond to arbitrary prime ideals.

\begin{defn}[Big Verma modules]
Let $\p \in \spec R$ such that $v \in \p$ and set $A^\p = A \tensor_R R/\p = A/A\p$.  Then $A^{\p,-1} = \bigoplus_{i < 0} A^\p_i$ is a submodule since $\tp \tm = v = 0$ in $R/\p$.  Set $V^\p = A^\p / A^{\p,-1}$ and observe that $V^\p$ is a non-negatively graded module.  We define $\chi_\p = \{ k \in \Z : k > 0, \sg^k(v) \in \p \}$ and $\chi'_\p = \chi_\p \cup \{\infty\}$.
\end{defn}

The big Verma modules are spread more evenly through $\Op$.  We make this precise in Proposition \ref{filt}.

\begin{lemma}
Let $k \geq 0$.  $V^{\p,k}= \bigoplus_{i \geq k} V^\p_i$ is a submodule of $V^\p$ if and only if $k \in \chi_p$.
\end{lemma}
\begin{proof}
$V^{p, k}$ is a submodule of $V^\p$ if and only if $\tm (V^\p)_k = 0$.  Since $V^\p_k$ is generated by $\tp^k$ as an $R$ module and $\tm \tp^k = \sg(v) \tp^{k-1} = \tp^{k-1}\sg^k(v)$ we see that $\tm V^\p_k = 0$ if and only if $\sg^k(v) \in \p$.  
\end{proof}

\begin{defn}
It will be convenient to have uniform notation for certain quotients.  Let $k \in \chi'_\p$ and set $Q^{\p,k} = V^\p/V^{\p,k}$ for $k < \infty$ and $Q^{\p,\infty} = V^\p$.   Note that $\supp( Q^{\p,k}[n] ) = Z(\sg^{-n}(\p)) = \Sg^n Z(\p)$ and this is a subset of $Z^\sg$ that is closed in $\spec R$. 
\end{defn}

\begin{prop} \label{filt}
A finitely generated, graded module $M$ in $\O^+$ has a finite filtration where each successive quotient has the form $Q^{\p,k}[n]$ for some nonzero prime $\p \in \spec R, k \in \chi'_\p$, and $n \in \Z$.
\end{prop}
\begin{proof}
Because $A$ is Noetherian, it suffices to show that any finitely generated module in $\O^+$ has a submodule of the desired type.  Let $k$ be the smallest degree such that $M_k \neq 0$.  Since $R$ is Noetherian, as a right module, $M_k$ has an associated prime, which we write as $\sg^k(\p)$ for convenience.  This prime must be nonzero since $v M_k = \tp\tm M_k \subset \tp M_{k-1} = 0$ so $v \in \sg^k(\p)$.  Let $x \in M_k$ be an element such that $\ann(x) = \sg^k(\p)$.  Then $A \cdot x[k] \cong V^\p / J$ where $J \subset V^\p$ is a submodule such that $J_0 = 0$.  If $J = 0$ then $A \cdot x \cong Q^{\p,\infty} [-k]$ has the desired form.  Otherwise, let $l = \min \delta(J)$.  Observe that the maps $V^\p_i \to V^\p_{i+1}$ given by $x \mapsto \tp x$ are all isomorphisms.  Let $y \in V^\p_0$ be such that $\tp^l y \in J_l$.  Then the submodule of $V^\p / J$ generated by $y$ is isomorpic to $V^\p/V^{\p,l}$.  This corresponds to a submodule of $A \cdot x$ isomorphic to $Q^{\p,l}[-k]$.
\end{proof}

Suppose that $M$ is a finitely generated graded $A$ module belonging to $\O^+$.  Then $M$ has a finite filtration such that the associated graded is isomorphic to $\oplus_{j=1}^N{ (V^{\p_j}/V^{\p_j,k_j})[n_j]}$ where $\p_j \in \spec R$ and $k_j \in \chi'_\p$.  Therefore $\supp(M) = \bigcup_{j=1}^N{ \Sg^{n_j}(Z(\p_j))}$ and this is a closed subset of $\spec R$ contained in $Z^\sg$.  $Z^\sg$ could have many connected components and $M$ decomposes into summands supported on each of the connected components of $Z^\sg$.  We will introduce an assumption that makes sure $Z^\sg$ has many connected components that are easy to describe.  From now on, we assume that the following property ($\ast$) holds.
\begin{assumption}
 ($\ast$) For every connected component $W$ of $Z$ and every $n$, $W \cap \Sg^n(Z)$ is either empty or is a connected component of $Z$ and for $n \gg 0$, $W \cap \Sg^n(Z) = \empt$.
\end{assumption} 

\begin{example}
There are two examples of GWA data satisfying ($\ast$), inspired by geometric considerations.  First, let $R$ be a ring let $f_1,\dotsc,f_k \in R$ be any invertible elements.  Let $\sg$ be the automorphism of $R[h_1,\dotsc,h_k]$ be defined by $\sg(h_i) = h_i + f_i$.  Suppose that $s_1,\dotsc,s_k,r_1,\dotsc,r_n \in R$ be such that $\sum_{i=1}^k{f_i s_i}$ is a unit in $R$ and let $v(h_1,\dotsc,h_k) = \prod_{i=1}^n{ (\sum_{j=1}^k{s_j h_j} - r_i) }$.  Then the GWA data $(R[h_1,\dotsc,h_k], \sg, v(h_1,\dotsc,h_k))$ satisfies ($\ast$).  Geometrically, $R[h_1,\dotsc,h_k]$ is the total space of the trivial rank $k$ vector bundle on $\spec R$, $\bar{f} = (f_1,\dotsc,f_k)$ is a section, and $\sg$ is translation by this section.  To see $v$ geometrically, let $\phi$ be a vector bundle map from $\spec R \times \mathbb{A}^k \to \spec R \times \mathbb{A}^1$ that restricts to an isomorphism on the subbundle spanned by $(f_1,\dotsc,f_k)$.  Then $v$ is the pullback along $\phi$ of a function on $\spec R \times \mathbb{A}^1$ that does not vanish on any fiber.  This class of examples includes the classical case where $R = \C[h]$, $k=1$, and $f_1 = 1$ so that $\sg(h) = h+1$ and $v \in \C[h]$ is some polynomial. 
\end{example} 

\begin{example}
We can construct another class of examples using dilation.  Let $R_\bt = \bigoplus_{i \geq 0} R_i$ be a graded ring generated by $R_1$ over $R_0 = \C$.  For $\gamma \in \C$, not a root of unity, let $\sg$ be defined on homogeneous $x$ by $\sg(x) = \gamma^{\deg(x)}x$.  Then for $v = h - 1$ where $h \in R_1$, the GWA data $(R_\bt,\sg,v)$ will satisfy ($\ast$).  Of course, $\spec R_\bt$ embeds in $\spec \sym_\C R_1$ as a cone and the action of $\sg$ is induced by dilation by $\gamma$.  Note that in this case $\spec R/(v) \subset \spec R$ is an affine hyperplane section that corresponds to a dense affine open in $\proj R_\bt$.
\end{example}

We will now construct a system of projective generators for $\Op$.  Notice that if $M$ is in $\Op$ then any map $A \to M$ factors through a map $A/A\tm^n \to M$ because if $m$ is the image of $1$ then by assumption $\tm^n m = 0$ for $n \gg 0$.  Let $Z^\sg_+ = \bigcup_{n \geq 0}{\Sg^n(Z)} \subset Z^\sg$.

\begin{defn}For each $n > 0$ let $A(n) = A/A\tm^n$.  Recall that this is naturally an $(A,R)$ bimodule.  We calculate that
\begin{equation}\label{pieces}
 (A\tm^n)_i = 
\begin{cases}
	A_i & i \leq -n, \\
	\tp^i \cdot R \prod_{j=i}^{n-1} {\sg^{-j}(v)} & -n < i \leq 0, \\
	\tm^i \cdot R \prod_{i=0}^{n-1}{\sg^{-j}(v)} & i \geq 0.
\end{cases}
\end{equation}
Hence $\ann( A(n) ) = ( \prod_{j=0}^{n-1}{\sg^{-j}(v)} )$ and $\supp( A(n) ) = \bigcup_{j=0}^{n-1}{\Sg^n(Z)} \subset Z^\sg_+$.  By ($\ast$), the connected components of $\supp( A(n) )$ are $\Sg$ translates of the components of $Z$.  Let $\pi_0(Z)$ and $\pi_0(Z^\sg)$ denote the sets of connected components of $Z$ and $Z^\sg$ respectively.  For each connected component  $W \in \pi_0(Z^\sg)$ let $A(n)_W$ be the summand of $A(n)$ supported on $W$.  
\end{defn}

We now consider three maps.  First we have the natural quotient map $q: A(n+1)  \to A(n)$.  Second and third we have the maps $\phi^+:A(n) \to A(n+1)$ and $\phi^-:A(n+1) \to A(n)$ induced by the endomorphisms $\phi^+(x) = x\tp$ and $\phi^-(x) = x\tm$ of $A$.

The quotient restricts to a surjection $q_W:A(n+1)_W \onto A(n)_W$.  If $W \cap \Sg^n(Z) = \empt$ then $q_W$ is an isomorphism.  The kernel of $q_W$ is the summand of $A\tm^n/A\tm^{n+1}$ that is supported on $W$.  Using \eqref{pieces} we calculate that
\begin{equation}\label{Anidealcomputation}
(A\tm^n/A \tm^{n+1})_i \cong 
\begin{cases}
0 & i \leq -(n+1) \\
R/\left( (\sg^{-n}(v)) + \ann(\prod_{j=i}^{n-1}{\sg^{-j}(v)})\right) & -n \leq i \leq 0 \\
R/\left( (\sg^{-n}(v)) + \ann(\prod_{j=0}^{n-1}{\sg^{-j}(v)}) \right) & i \geq 0
\end{cases}
\end{equation}
and conclude that $\supp (A\tm^n/A\tm^{n+1}) \subset \Sg^n(Z)$.  Hence ($\ast$) implies that for $n \gg 0$, $q_W$ is an isomorphism.  Let $N \gg 0$ be so large that for all $n \geq N$, $W \cap \Sg^n(Z) = \empt$ and define $A_W = A(N)$.  The previous discussion justifies the notation since if we were to choose a different $N$ with the same property there would be a canonical isomorphism between the resulting $A_W$'s.

The maps $\phi^+,\phi^-$ induce arrows $\phi^+_W:A(n)_{\Sg(W)} \to A(n+1)_W[1]$ and $\phi^-_W: A(n+1)_W[1] \to A(n)_{\Sg(W)}$.  By construction we have $\phi^+_W \circ \phi^-_W(x) = x\sg(v) = x \cdot v$ and $\phi^-_W \circ \phi^+_W(x) = x \cdot v$.  If $Z \cap \Sg(W) = \empt$ then $v$ acts invertibly on $A(n)_{\Sg(W)}$ and $A(n+1)_W[1]$, since $\supp(A(n+1)_W[1]) = \Sg(W)$.

Let $W \in \pi_0(Z^\sg_+)\setminus \pi_0(Z)$ and consider $A_W$.  It follows from ($\ast$) that $W = \Sg^k(W_0)$ for some $W_0 \in \pi_0(Z)$ and $k > 0$.  We may assume that $\Sg^n(W_0) \cap Z = \empt$ for $0 < n \leq k$.  Let $N \gg 0$ be so large that $A_{W_0} = A(N)_{W_0}$ and $A_{W} = A(N+k)_W$.  Then $\phi^+$ induces a string of isomorphisms
\[ 
A_W = A(N)_{\Sg^k(W_0)} \iso A(N+1)_{\Sg^{k-1}(W_0)}[1] \iso \dotsb \iso A(N+k)_{W_0}[k] = A_{W_0}[k]
\]
We conclude that the collection of $A_W$, $W \in \pi_0(Z^\sg)$ can be obtained as shifts of a finite set of modules parameterized by $\pi_0(Z)$.  
\begin{defn}[Projective generators]
For each $W \in \pi_0(Z)$ let $P_W = A_W$.  
\end{defn}
Observe that, by construction, $(P_W)_0$ is an indecomposable right $R$ module.  Because $P_W$ is generated in degree zero, this implies that it is an indecomposable graded $A$ module.

\begin{prop} \label{pgen}
The set $\{P_W[n]\}$ where $W\in \pi_0(Z)$ forms a system of projective generators for $\O^+$.  The set $\{P_W\}$ where $W \in \pi_0(Z)$ also forms a system of projective generators for $\O$.
\end{prop}
\begin{proof}
Let $M \onto M''$ be a surjection of modules in $\O^+$ and let $P_W[n] \to M''$.  We will show that there is a lift $P_W[n] \to M$.  It will be convenient to replace this problem with the equivalent problem obtained by applying $[-n]$ to the maps and modules.  Now, write $P_W = A(N)_W$ for some $N \gg 0$.  The composite $A \to A(N) \to A(N)_W \to M''$ lifts to a map $f:A \to M$.  If $-m < \min(\delta(M))$, the map $f$ factors as $A \to A(m) \to M$.  We can assume that $m > N$.  There is a commutative square
\[
\xymatrix{
A(m)  \ar[r]^q \ar[d] & A(N)_W \ar[d] \\
M \ar[r] & M' \\
}
\]
Note that $q:A(m)_W \to A(N)_W$ is an isomorphism.  Resticting the map $A(m) \to M$ to a map $A(m)_W \to M$ and composing with $q$ gives the desired lift $P_W \to M$.

As we noticed in the previous paragraph, every map $A \to M$ with $M$ an object of $\O^+$ factors through a map $A(n) \to M$ for some $n$.  Hence, the collection $A(n)[m]$ is a system of generators for $\O^+$.  But $A(n)$ is a finite direct sum of $A(n)_{W'}$ where $W'$ ranges over $\pi_0(Z^\sg)$.  Each of the $A(n)_{W'}$ is a quotient of $P_W$ for some $W \in \pi^0(Z)$ such that $\Sg^k(W) = W'$.  Therefore the $P_W[m]$ also form a generating set.

Forgetting the gradings, the same argument shows that $P_W$ are projective in $\O$ and generate since the $A(n)$ do.
\end{proof}

\begin{example}  
Let us see what happens when ($\ast$) is not satisfied.  We take $R = \C[x,y]$, $v = x$, and $\sg$ an irrational rotation of the plane, i.e. a rotation of infinite order.  Consider the corresponding GWA.  We will see that $\O^+$ does not have enough projectives.  Consider the surjections $A(n+m) \onto A(n)$, suppose that $P$ is a projective in $\O^+$ with a map $P \to A(n)$.  For every $m$ there is a lift $P \to A(n+m)$ that completes the commutative diagram
\[
\xymatrix{
 & P \ar[d]\ar[dl] \\
A(n+m) \ar[r] & A(n) \\
}
\]
Fix a surjection $\oplus_i{ A(n_i)[k_i] } \onto P$ where $i$ may range over an arbitrary index set.  Let us calculate the possibilities for the image of a map $A(n_i)[k_i] \to A(n)$ that factors through the quotient $A(n+m) \to A(n)$.  Since $A(n_i)[k_i]$ is cyclic, $\Homgr(A(n_i)[k_i], A(n+m))$ can be identified with the set of elements $x \in A(n+m)_{-k_i}$ such that $\tm^{n_i} x = 0$.  Let $\la_0 = (x,y)$ and for $X \subset \Z$ a nonempty finite subset let $p_X = \prod_{j \in X}{ \sg^j(v) }$ and $q_X = \prod_{j \in X}{\sg^{-j}(v)}$.  For $l_1 < l_2$ integers write $[l_1,l_2] = \{l_1,l_1+1,\dotsc,l_2-1,l_2\}$.  Notice that $p_X,q_X \in \la_0$ for any $X$.  Assume that $n+m > \max\{n_i + k_i, -k_i - n_i\}+1$.  First, suppose $k_i > 0$.  Then $x = \tm^{k_i} r$ and $\tm^{n_i} x = \tm^{n_i+k_i} r =0$ if and only if $r \in (q_{[n_i+k_i,n+m]})$.    Then the image under composition with the quotient $A(n+m) \onto A(n)$ is the submodule generated by $\tm^{k_i} r$ which is contained in the submodule generated on the right by $q_{[n_i+k_i,n+m]}$.  We conclude that in this case the image is contained in $A(n)\la_0$.  On the other hand, if $k_i < 0$ then $x = \tp^{-k_i} r$ and 
\[ 
\tm^{n_i} \tp^{-k_i} r = 
\begin{cases}
\tp^{-k_i - n_i} p_{[-k_i - n_i,n + m]} r & -k_i \geq n_i, \\
\tm^{n_i + k_i} p_{[n_i + k_i,n + m]} r & n_i + k_i \geq 0.
\end{cases}
\]
These are zero in $A(m+n)$ if and only if $p_{[-k_i - n_i,n+m]}r \in (q_{[0,n+m-1]})$ or $p_{[n_i+k_i,n+m]}r \in (q_{[n_i+k_i,n+m]})$ respectively.  Both situations imply that $r \in (q_{[1,n+m-1]})$ and therefore that the image $A(n_i)[k_i] \to A(n)$ is contained in $A(n)\la_0$.  We conclude that for each $i$ there is an $m$ such that if $A(n_i)[k_i] \to A(n)$ factors through $A(n+m) \to A(n)$ then the image of the map is contained in $A(n)\la_0$.  By assumption, for each $i$, the map $A(n_i)[k_i] \to P \to A(n)$ factors through $A(n+m)$ for every $m$ and thus the image of this map is contained in $A(n)\la_0$.  Therefore the image of $P$ in $A(n)$ is contained in $A(n)\la_0$ and we conclude that there is not a system of projective generators for $\O^+$.
\end{example}

\begin{prop} \label{phom}
If $W_1,W_2 \in \pi_0(Z)$ then $\Homgr(P_{W_1},P_{W_2}[n]) = 0$ unless $W_1 = \Sg^n(W_2)$.
\end{prop}
\begin{proof}
By construction $\supp(P_{W_1}) = W_1$ and $\supp( P_{W_2}[n] ) = \Sg^n(W_2)$.  By ($\ast$), either $W_1 = \Sg^n(W_2)$ or else $W_1 \cap \Sg^n(W_2) = \empt$.  Clearly if $W_1 \cap \Sg^n(W_2) = \empt$ then $\Homgr( P_{W_1},P_{W_2}[n] ) = 0$. 
\end{proof}

Note that $\Hom_A(P_{W_1},P_{W_2}) = \bigoplus_{n \in \Z} \Homgr(P_{W_1},P_{W_2}[n])$.  By Proposition \ref{phom}, if more than one of these spaces is nonzero then $W_1 = \Sg^{n_1}(W_2) = \Sg^{n_2}(W_2)$ and thus $W_2 = \Sg^{n_1-n_2}(W_2)$.  But this implies that $\Sg^{m(n_1 - n_2)}(W_2) \cap Z \neq \empt$ for all $m$, contradicting ($\ast$).  Therefore every map in $\Hom_A(P_{W_1},P_{W_2})$ is automatically homogeneous of some particular degree.

\begin{defn}
Let $\pi_0(Z)/\Sg$ be the set of equivalence classes in $\pi_0(Z)$ for the equivalence relation $W_1 \sim W_2$ if there is an $n$ such that $W_2 = \Sg^n(W_1)$.  For each $w \in \pi_0(Z)/\Sg$, fix $W_w \in w$ and set $\chi_w = \{n \in \Z : \Sg^n(W_w) \in \pi_0(Z)\}$.  For each $n \in \chi_w$ set $P_{w,n} = P_{\Sg^{-n}(W_w)}[n]$.  Finally, for $w \in \pi_0(Z)/\Sg$ let $\Op_w$ be the thick subcategory generated by the projective modules $P_{w,n}$, $n \in \chi_w$.  Define $\Op_Z$ to be the thick subcategory of $\Op$ generated by all of the $\Op_w$.  Note that the various $\Op_w$ are \emph{not} closed under shifting.  
\end{defn}

\begin{defn} If $\mc{A}$ is an abelian category then $\oplus_\Z \mc{A}[n]$ is the category whose objects are formal sums of formal shifts $\oplus_i{a_i[n_i]}$ for $a$ and object of $\mc{A}$ and $n_i \in \Z$ and where 
\[
\Hom(\bigoplus_i{ a_i[n_i]},\bigoplus_j{b_j[m_j]}) = \prod_i \bigoplus_j \Hom(a_i[n_i],b_j[m_j]),
\]
and
\[ 
\Hom(a[n],b[m]) = 
\begin{cases}
	\Hom_\mc{A}(a,b) 	& n=m \\
	0 			&	\text{otherwise}
\end{cases}.
\]
\end{defn}

\begin{thm} \label{gr}
The functor which forgets the grading defines an equivalence between $\Op_Z$ and $\O$.  Moreover, $\Op \cong  \bigoplus_{w \in \pi_0(Z)/\Sg} \bigoplus_{n \in \Z} \Op_w[n]$.
\end{thm}
\begin{proof}
The forgetful functor $\Op_Z \to \O$ is exact and therefore fully faithful because we constructed the $P_{w,n}$ so that $\Hom(P_{w_1,n_1},P_{w_2,n_2}) = \Homgr(P_{w_1,n_1},P_{w_2,n_2})$.  By Proposition \ref{pgen}, the essential image contains projective generators and therefore the forgetful functor is an equivalence.  

To check the second assertion it suffices to show that $\Homgr(P_{w,n},P_{w',n'}[m]) = 0$ unless $w = w'$ and $m = 0$.  If $w \neq w'$ then $\supp(P_{w,n}) \cap \supp(P_{w',n'}[m]) = \empt$ and therefore $\Homgr(P_{w,n},P_{w',n'}[m]) =0$.  Assume that $w = w'$.  By definition $P_{w,n} = P_{\Sg^{-n}(W_w)}[n]$ and $P_{w,n'}[m] = P_{\Sg^{-n'}(W_w)}[n'+m]$.  Now, Proposition \ref{phom} implies that $\Homgr(P_{\Sg^{-n}(W_w)}[n],P_{\Sg^{-n'}(W_w)}[n'+m]) = \Homgr(P_{\Sg^{-n}(W_w)}, P_{\Sg^{-n'}(W_w)}[n'-n+m]) = 0$ unless $W_w = \Sg^m(W_w)$.  By ($\ast$), if $W_w = \Sg^m(W_w)$ then $m=0$.
\end{proof}

We cannot make this decomposition into thick subcategories finer.  Indeed suppose that $W_1 = \Sg^n(W_2)$ and that $n > 0$.  Then the map $P_{W_1} \to P_{W_2}[n]$ induced by right multiplication by $\tp^n$ on $A$ is nonzero.   However, in general it will neither be injective or surjective.

\begin{cor}
Set $P = \bigoplus_{W \in \pi_0(Z)}{P_W}$.  $P$ is a projective generator of $\O$.  Hence, $\O$ is equivalent to the category of right modules over the finite $R/(v)$ algebra $\End_A(P)$.
\end{cor}
\begin{proof}
We just need to check that $\End_A(P)$ is finite.  However it follows from the preceding paragraphs that
\[
\End_A(P) = \bigoplus_{\substack{w,w' \in \pi_0(Z)/\Sg \\ n \in \chi_w, n' \in \chi_{w'}}}{\Homgr(P_{w,n},P_{w',n'})}.
\]
This is a finite direct sum and $\Homgr(P_{w,n},P_{w',n'}) \subset \Hom^{left}_R( (P_{w,n})_0, (P_{w',n'})_0)$, where $\Hom^{left}_R$ is the set of homomorphisms of \emph{left $R$-modules}.  Since $R$ is noetherian and $(P_{w,n})_0, (P_{w',n'})_0$ are cyclic left $R/(v)$ modules it follows that $\End_A(P)$ is a finitely generated module over $R/(v)$.  
\end{proof}

Observe that there is a map $\pi_0(Z)/\Sg \times \Z \to \pi_0(Z^\sg)$ defined by $(w,n) \mapsto \Sg^n(W_w)$.  By ($\ast$), this is a bijection.  So we can also think of this above decomposition as parameterized by $\pi_0(Z^\sg)$.  Let $W \in \pi_0(Z^\sg)$.  We know that $A_W \cong P_{W'}[k]$ and by ($\ast$), there is a unique $n$ such that $W' = \Sg^n(W)$.  Therefore $P_{W'} = P_{w,n}[-n]$ and it follows that $A_W \cong P_{w,n}[k-n]$ belongs to $\Op_w[k-n]$.  So unfortunately, the parameterization of the decomposition of $\Op$ disagrees with our parameterization of the fundamental modules $A_W, W \in \pi_0(Z^\sg)$.  Finally, we note that if ($\ast$) holds then Theorem \ref{gr} implies that Proposition \ref{filt} applies to $\O$ as well as $\O^+$.

It follows from \ref{simp} that ($\ast$) implies that for each $\la \in \maxspec R$, if $\chi_\la \neq \empt$ then it is finite so $A^\la$ is Artinian.  We introduce uniform notation by setting 
\[
 A^{\la,-} =
 \begin{cases}
 	A^\la & \chi_\la = \empt \\
 	A^{\la,\min \chi_\la} & \chi_\la \neq \empt, \min \chi_\la \leq 0 \\
 	A^\la/A^{\la,\min \chi_\la} & \chi_\la \neq \empt, \min \chi_\la > 0
 \end{cases}.
\]
Note that $A^{\la,-}$ is simple and for $n \ll 0$, $n \in \delta(A^{\la,-})$.  The following Lemma will be used in the next section.

\begin{lemma} \label{crito}
The modules in $\O^+$ are exactly those graded $A$ modules not having any $A^{\la,-}$ as a subquotient.
\end{lemma}
\begin{proof}
First, note that $\O^+$ is closed under taking subquotients.  Therefore if $M$ belongs to $\Op$ then all simple subquotients belong to $\Op$.  A simple module belongs to $\Op$ if and only if it is not of the form $A^{\la,-}$ for any $\la \in \maxspec R$.

Now, if $M$ does not belong to $\Op$ then there is some homogeneous $m$ such that $Am$ does not belong to $\Op$.  Of course, $A m \cong A/J[n]$ and $A/J$ is not in $\O^+$ if and only if $A/J[n]$ is not in $\Op$.  Since the set of modules of the form $A^{\la,-}$ is closed under the shift, $A^{\la,-} \cong A^{\sg(\la),-}[1]$, it suffices to show that if $A/J$ is not in $\Op$ then $A/J$ has a subquotient of the form $A^{\la,-}$.  As remarked before $J$ is automatically a right $R$ submodule of $A$.  Because $R$ is noetherian, for $n \gg 0$ we have $\tm J_n = J_{n-1}$.  Hence there is an ideal $\bar{J}$ such that for $n \ll 0$, $(A/J)_n \cong R/\bar{J}$ as a right $R$ module.  Since $(A/J)_n \neq 0$ for $n \ll 0$, there is a $\la \in \maxspec R$ such that $\bar{J} \subset \la$.  Therefore $(A/J \tensor_R R/\la)_n \neq 0$ for $n \gg 0$.  But since $M' = A/J \tensor_R R/\la$ is a quotient of $A_\la$ such that $\delta(M')$ is not bounded below, it must contain $A^{\la,-}$ either as a submodule or as a quotient.
\end{proof}

\section{Graded Morita equivalence}

For a graded ring $A$ let $A-\grMod$ be the category of graded left $A$ modules.  This category is equipped with an auto-equivalence $(-)[1]$, the usual shift.  
\begin{defn}
Let $A$ and $B$ be graded rings. A \emph{strongly graded Morita equivalence} is a $\C$-linear equivalence of categories $F:B-\grMod \iso A-\grMod$ such that for any graded $B$ module $M$, $F(M[1])\cong F(M)[1]$.  
\end{defn}
Since $B$ is projective in $B-\grMod$, $P=F(B)$ is projective in $A-\grMod$.  The monomorphisms in $A-\grMod$ are exactly the injective module maps.  Therefore we can detect whether or not an object satisfies the ascending chain condition using only the abstract structure of the category $A-\grMod$.  Suppose that $B$ is noetherian.  Then the graded $B$ module $B$ satisfies the ascending chain condition.  Therefore $P$ also satisfies the ascending chain condition and we conclude that $P$ is finitely generated.  Note that $\Homgr(B,B[n]) = B_n$ and composition $\Homgr(B,B[n])\tensor \Homgr(B[n],B[n+m]) \to \Homgr(B,B[n+m])$ is identified with multiplication $B_n \tensor B_m \to B_{n+m}$ so that $a \circ b$ is identified with $ba$ in $B$.  Therefore we can think of $P$ as a right graded $B$ module.  The important thing is that the single grading makes $P$ both a graded $A$ module and a graded $B$ module.  For any graded $B$ module $M$, $F(M) \cong P \tensor_B M$.  In this section, we will study the notion of strongly graded Morita equivalence for GWAs.

Now, consider a map rings $f:R \to S$ and suppose that $\sg_R \in \Aut(R)$ and $\sg_S \in \Aut(S)$.  Say that $f$ is \emph{$\sg$ equivariant} if $\sg_S \circ f = f \circ \sg_R$.  Given GWA data $(R,\sg,v)$ and a $\sg$ equivariant automorphism $\psi \in \Aut(R)$ we can construct an isomorphism $\Psi:T(R,\sg,v) \to T(R,\sg,\psi(v))$ extending $\psi$ and satisfying $\Psi(\tp) = \tp$ and $\Psi(\tm) = \tm$.  We can view a graded $T(R,\sg,\psi(v))$ module as a graded $T(R,\sg,v)$ module through $\Psi$ and this sets up a strongly graded Morita equivalence between these two GWAs.  We will see that any strongly graded Morita equivalence leads to an equivariant isorphism of the ground rings.

\begin{thm} \label{mornec}
For $j =1,2$ let $(R_j,\sg_j,v_j)$ be GWA data.  Assume that there is some $\la \in \maxspec R_j$ such that $\la \cap \{ \sg_j^n(v_j) : n \in \Z\} = \empt$.  Suppose that $T(R_1,\sg_1,v_1)$ and $T(R_2,\sg_2,v_2)$ are strongly graded Morita equivalent.  Then there is an equivariant isomorphism $\rho:R_1 \to R_2$ such that $\la \in \maxspec R_2$ contains a $\sg_2$ translate of $v_2$ if and only if $\rho^{-1}(\la)$ contains a $\sg_1$ translate of $v_1$.  Moreover, the equivalence restricts to an equivalence between $\O^+_1$ and $\O^+_2$.
\end{thm}
\begin{proof}
Set $A = T(R_1,\sg_1,v_2)$ and $B = T(R_2,\sg_2,v_2)$.  Let $P$ be the finitely generated, projective, graded $A$ corresponding to $B$ under a strongly graded Morita equivalence $F:B-\grMod \to A-\grMod$.  Note that since $P$ is a summand of $\oplus_j{ A[n_j] }$ each graded piece $P_i$ is a projective left $R_1$ module.  Recall that $P$ is an $(A,R_1)$ bimodule such that for any $p \in P_i, r \in R_1$ we have $r p = p \sg^i(r)$.  Thus $P_i$ has the same rank as a left module and as a right module.  For now, assume that for each $P_0$ has rank 1.  Then $\End_{R_1}(P_0) = P_0^{\vee} \tensor_{R_1} P_0 \iso R_1$.  Note that $\Homgr(P,P[n]) = \Homgr(B,B[n]) = B_n$ and in particular $\Endgr(P) = B_0 = R_2$.  Recall that every degree preserving map of graded $A$ modules respects the $(A,R)$ bimodule structure.  Therefore we can define maps $\rho:R_1 \to \Endgr(P) = R_2$ by $\rho(r)(p) = pr$ and $\phi:\Endgr(P) \to \End_{R_1}(P_0) \iso R_1$ by restriction.  By definition, for $p \in P_0$ and $f \in \Endgr(P)$ we have $f(p) = p\phi(f)$.  Let us check that $\rho \circ \sg^{-1}_1 = \sg^{-1}_2 \circ \rho$.  First note that if $x \in \Endgr(P,P[1])$ corresponds to $\tp$ in $B_1$ then for any $f \in \Endgr(P)$ we have $f[1] \circ x = x \circ \sg^{-1}_2(f)$ and $x \circ f = 0$ if and only if $f = 0$.  Hence it suffices to show that $x \circ \rho(\sg^{-1}_1(r)) = \rho(r)[1] \circ x$.  Observe that $\rho(r)[1]$ is not right multiplication by $r$ but instead right multiplication by $\sg_1^{-1}(r)$.  Now for any $r \in R_1$ and $p \in P$ we compute
\[
(x \circ \rho(\sg^{-1}(r)))(p) = x( p \cdot \sg^{-1}(r) ) = x(p) \cdot \sg^{-1}(r) = (\rho(r)[1] \circ x)(p)
\]
We conclude that $\rho\circ \sg_1 = \sg_2 \circ \rho$.

Clearly $\phi \circ \rho = \id_{R_1}$.  Hence, $\rho$ is injective and $\phi$ is surjective.  Interchanging the roles of $A$ and $B$ we obtain $\rho':R_2 \to R_1$ and $\phi':R_1 \to R_2$, injective and surjective respectively.  Now any surjective ring endomorphism of a noetherian ring is automatically an automorphism.  Thus $\phi' \circ \phi$ is an isomorphism.  This implies that $\phi$ is injective and it follows that $\rho$ and $\phi$ are inverse isomorphisms.  Therefore $\rho$ is the equivariant isomorphism that we wanted.

Next we will show that $P_0$ does indeed have rank 1, using the Lemmata from \S 2.  The hypothesis and Lemma \ref{simp} imply that there is a $\la \in \maxspec R_2$ such that $B^\la$ is simple.  As mentioned in Remark \ref{notinzsg}, if $T$ is a simple $B$ module such that $\Extgr^1(T,B^\la) \neq 0$ then $B^\la \cong T$.  Therefore $F(B^\la)$ is a simple module with the same property, and so we must have $F(B^\la) = A^\mu$ for some $\mu$ such that $A^\mu$ is simple.  

Fix a presentation $\oplus_j A[n_j] \onto P$.  Now $A[n_j] \tensor R/\mu = A^{\sg^{n_j}(\mu)}[n_j] \cong A^\mu$ by Lemma \ref{gr}.   Hence we get a surjection $\oplus_j A[n_j] \tensor_{R_1} R/\mu = \oplus_j A^\mu \onto P/P\mu$ is a surjection.  Since $A^\mu$ is simple, this implies that $P/P\mu \cong (A^\mu)^{\oplus m}$ where $m = \dim_\C P_0/P_0\mu$.  Under the Morita equivalence, this corresponds to a surjection $B \onto (B^\la)^{\oplus m}$.  However, up to scaling there is only one graded map $B \to B^\la$, and therefore any map $B \to (B^\la)^{\oplus m}$ factors as $B \onto B^\la \to (B^\la)^{\oplus m}$, which is not surjective unless $m = 1$.  We conclude that $m = \dim_\C P_0/P_0\mu = 1$ so $P_0$ has rank 1 as a projective module.  

Now, set $\mu = \rho^{-1}(\la)$.  As we argued above, if $\la$ does not contain a $\sg_2$ translate of $v_2$ then $B^\la$ is simple and $F(B^\la) = P/\la(P) = P/P\mu \cong A^\mu$ is simple, so $\mu$ does not contain a $\sg_1$ translate of $v_1$.

We must argue that $F$ preserves $\Op$.  Note that for a graded $B$ module $M$, we compute  $\delta(M) = \{n \in \Z : \Homgr(B,M[n]) \neq 0\}$.  Therefore $\delta(M) = \{ n \in \Z : \Homgr(P,F(M)[n])\neq 0\}$.  Now, among graded simple modules, those of the form $B^{\la,-}$ are characterized by the property that $\delta(B^{\la,-})$ is not bounded below.   Let $S$ be a simple graded $A$ module.  Returning to our presentation $\oplus_j A[n_j] \onto P$ we see that $\Homgr(P,S[n]) \subset \oplus_j \Homgr(A[n_j],S[n]) \cong S_{n-n_j}$ as $\C$ vector spaces.  So $\{ n \in \Z : \Homgr(P,S[n]) \neq 0\}$ is unbounded below if and only if $\delta(S[n])$ is not bounded below if and only if $S[n]$ is of the form $A^{\mu,-}$ for some $\mu \in \maxspec R_1$.  It follows that for every $\la \in \maxspec R_2$ there is a $\mu \in \maxspec R_2$ such that $F(B^{\la,-}) \cong A^{\mu,-}$.   By \ref{crito}, if $M$ is not in $\Op_2$ then $M$ contains some $B^{\la,-}$ as a subquotient.  But then $F(M)$ contains some $A^{\mu,-}$ as a subquotient, so $F(M)$ is not in $\Op_1$.  Applying the same reasoning to an inverse equivalence, we see that $F$ restricts to an equivalence $\Op_2 \to \Op_1$.
\end{proof}

Recall that the classical GWAs are defined by data $(\C[h],\tau,v)$ where $\tau(p(h)) = p(h+1)$.  In an article by Bavula and Jordan we find the following theorm \cite[Theorem 3.8]{BJ} concerning the isomorphism problem.
\begin{thm*}
Let $v_1,v_2 \in \C[h]$.  Then $T(v_1) \cong T(v_2)$ if and only if there exist $\eta,\nu \in \C$ with $\eta \neq 0$ such that $v_2(h) = \eta v_1(\nu \pm h)$.
\end{thm*}   
One consequence of Theorem \ref{mornec} (or Hodges \cite[Lemma 2.4]{H}) is that $T(\C[h],\sg, h(h+1))$ and $T(\C[h],\sg,h(h+2))$ are strongly graded Morita equivalent but by the theorem above they not isomorphic.

Now we are going to develop a technique to produce strongly graded Morita equivalences between GWAs.  First we define an algebra.  Consider the oriented cycle $Q$ of length $n>0$ viewed as a quiver.  Let $I$ be the vertex set with an action of $\Z$ generated by the automorphism which sends a vertex to the next vertex along the cycle, written as $i \mapsto i+1$.  Let $a_i$ be the edge joining $i$ to $i+1$.  Form the double quiver $\bar{Q}$ which is constructed from $Q$ by adding, for every edge $a \in Q$, a dual edge $a^*$ with the opposite orientation.  So $a_i^*$ joins $i+1$ to $i$.  Let $R$ be a finitely generated commutative $\C$ algebra.  We first form the path algebra $R\bar{Q}$.  Let $RI$ be the $R$ algebra generated by central orthogonal idempotents $\{1_i\}_{i \in I}$.  Let $R \bar{E}$ be the free symmetric $R$ bimodule generated by the edges of $\bar{Q}$.  $R \bar{E}$ has an $R I$ bimodule structure determined by the condition that $1_i e 1_j$ is equal to $e$ if $source(e) = i$ and $tail(e) = j$ and is zero otherwise.  The path algebra is defined by $R \bar{Q} = T^{\tensor}_{RI}{RE}$ and has an $R$ module basis identified with paths in $\bar{Q}$ as follows.  To a path $e_1 e_2 \dotsm e_n$ through $\bar{Q}$ we associate $e_1 \tensor e_2 \tensor \dotsb \tensor e_n \in (RE)^{\tensor n}$.  Now let $\hat{\sg}$ be an automorphism of $RI$ satisfying $\hat{\sg}(e_i) = e_{i+1}$.  Observe that $\hat{\sg}(\sum_{i \in I}{r_i e_i}) = \sum_{i \in I}{\sg_i(r_i) e_{i+1}}$ for some collection $\{\sg_i\}_{i \in I}$ of automorphisms of $R$.  So we can also think of $\hat{\sg}$ by assigning an automorphism $\sg_i$ of $R$ to each edge $a_i$ of $Q$.  
\begin{defn}
Given $\br = \sum_{i \in I}{r_i e_i} \in RI$ there is an algebra $\Pi = \Pi(R,\hat{\sg},\br)$ defined to be the quotient of the path algebra by the relations
\begin{equation}\label{biggwarelations}
 x a_i = a_i \sg_i(x), \quad \sg_i(x) a_i^* = a_i^* x, \quad a_i a_i^* = r_ie_i, \quad a_i^*a_i = \sg_i(r_i)e_{i+1},
\end{equation}
for any $x \in R$.  
\end{defn}
If $n = 1$ then $\Pi = T(R,\sg,\br)$.  For each $i \in I$, let $\Pi_i = e_i \Pi e_i$.  Note that both $\Pi$ and $\Pi_i$ naturally contain $R$ as a subring.  Moreover the path algebra has a natural grading with $\deg(RI) = 0, \deg(a) = \frac{1}{n},$ and $\deg(a^*) = - \frac{1}{n}$ for $a \in Q$.  The relations above are homogeneous and thus the grading descends to a grading on $\Pi$ and $\Pi_i$.

We define an automorphism $\theta_i$ of $R$ and an element $v_i$ by
\begin{align} \label{thetai}
 \theta_i & = \sg_{i-1} \circ \dotsm \circ \sg_{i+1} \circ \sg_i \\ \label{vi}
 v_i & = r_i \cdot \sg_i^{-1}(r_{i+1}) \cdot (\sg_i^{-1} \circ \sg_{i+1}^{-1})(r_{i+2}) \dotsm (\sg_i^{-1} \circ \dotsm \circ \sg_{i-2}^{-1})(r_{i-1}).
\end{align}
Informally, one obtains $\theta_i$ by composing the automorphisms $\sg_j$ in a circle starting at vertex $i$, and similarly one obtains $v_i$ by pulling back $r_j$ by the composition of the $\sg$'s on the backwards arc from $i$ to $j$ and multiplying all of these together.  There is a natural map $f:T(R,\theta_i,v_i) \to \Pi_i$.  We define the map on generators by $f(r) = r e_i$ for $r \in R$ and $f(\tp) = \ta := a_i a_{i+1} \dotsm a_{i-1}$ and $f(\tm) = \ta^* := a^*_{i-1} \dotsm a_i^*$.  It is easy to check that $f(\tp)$ and $f(\tm)$ satisfy the neccessary relations.  Therefore we get a map $f: T(R,\theta_i,v_i) \to \Pi_i$ which respects the natural grading on both sides.

\begin{lemma}\label{prez}
The natural map $T(R,\theta_i, v_i) \to \Pi_i$ is an isomorphism.
\end{lemma}
\begin{proof}[Proof sketch.]
Since $aa^*$ and $a^*a$ are in $RI$ for any edge $a$, the ring $\Pi_i$ is generated over $R$ by $\ta$ and $\ta^*$.  Therefore $f$ is surjective.  Note that $f$ is injective if and only if $R$ acts without torsion on $R\ta^k$ and $R(\ta^*)^k$ for any $k$. To prove this, we consider the twisted path algebra of $\bar{Q}$.  Let $(R\bar{E})_{\hat{\sg}}$ be the $RI$ bimodule obtained from $R E$ by redefining the left action by $\bm{a} \cdot e = \hat{\sg}(\bm{a})e$.  Then $S = T^\tensor_{R I}{ (R\bar{E})_{\hat{\sg}} }$ is the twisted path algebra.  As a left (and right) $R$ module $S \cong R \tensor_\C \C\bar{Q}$.  This algebra is like the path algebra except that instead of containing $R$ as a central subalgebra, we have $r a_i = a_i \sg_i(r)$ and $\sg_i(r) a^*_i = a^*_i r$.  Note that $S$ has a bigrading with $\deg(e_i) = (0,0), \deg(a_i) = (1,0)$, and $\deg(a_i^*) = (0,1)$. Set the total degree equal to the sum of the bidegrees. Let $x_i = a_i a_i^* - r_ie_i$ and $y_i = a_i^* a_i - \sg_i(r_i)e_{i+1}$.  Then $\Pi$ is a quotient of $S$ by the ideal $J = (x_i,y_i)_{i\in I}$.  Now we will show that if $s \ta^k \in J$ then $s = 0$ and similarly for $(\ta^*)^k$.

For paths $p,q$ say that $q < p$ if the total degree of $p$ is greater than the total degree of $q$ and $p$ is equal to a multiple of $q$ modulo $J$.  For $q < p$ there are elements $r(p,q) \in R$ such that $p = r(p,q)q$ modulo $J$ and if $q' < q < p$ then $r(p,q') = r(p,q)r(q,q')$.  Note that if $e_i p e_j = p$ then $e_i q e_j = q$ for all $q < p$.  Now, write $p = \sum_{q < p}{ s_q p}$ where $\sum{s_q} =1$.  Then $\sum_{q < p}{s_q p} = \sum_{q < p}{ s_q r(p,q) q} = s_{q_0} r(p,q_0)q_0$ if and only if for all $q \neq q_0$, $s_q r(p,q) = 0$.  Now $r(p,q_0) = r(p,q)r(q,q_0)$ so $s_q r(p,q_0) = 0$ for all $q_0 < q < p$ and $s_{q_0} = 1 - \sum_{q_0 \neq q < p}{s_q}$.  So if $q_0$ is minimal then $s_{q_0} r(p,q_0) = (1 - \sum_{q_0 \neq q < p}{s_q})r(p,q_0) = r(p,q_0)$.  The minimal $q_0$ for $p$ of degree $(nk + l,l)$ satisfying $e_i p e_i = p$ is $\ta^k$.  This means that given a path $p$ of degree $(nk+l,l)$ satisfying $e_i p e_i = p$, there is a well defined element $r(p,k) \in R$ such that if $p = s \ta^k$ mod $J$ then $s = r(p,k)$.  Suppose that $s\ta^k = 0$ modulo $J$.  Then we can write zero as $0 = \sum_{p,j}{ s_{p,j} p }$ such that for each path $p$, $\sum_j{s_{p,j}} = 0$ and $e_i p e_i = p$; and $s \ta^k = \sum_{p,j}{ s_{p,j}r(p,k)\ta^k }$.  But $\sum_{p,j}{s_{p,j} r(p,k)} = 0$ and therefore $s = 0$.  We deal with $(\ta^*)^k$ in a similar way.  Hence no multiples of $\ta^k$ or $(\ta^*)^k$ are zero in $\Pi$.  So we see that $f$ is injective as well.
\end{proof}

\begin{example} 
Let $\bm{\alpha} = \sum_{i \in I}{\alpha_i e_i} \in \C I$.  Then the deformed preprojective algebra of type $\bm{A}$ denoted $\Pi^{\bm{\alpha}}(Q)$ is the quotient of the path algebra $\C \overline{Q}$ by the relation $\sum_{i \in I}{[a_i,a^*_i]} - \bm{\alpha}$.  Take $R = \C[h]$, $\sg_i$ to be translation by $\alpha_{i+1}$, and $\br = h \cdot 1 = \sum_{i \in I}{he_i}$.  Then $S = \Pi(R,\hat{\sg},\br)$ is isomorphic to $\Pi^{\bm{\alpha}}(Q)$.  Indeed, there is an obvious map $f:\C \overline{Q} \to S$ given by $f(e_i) = e_i, f(a_i) = a_i$ and $f(a^*_i) = a_i$.  The defining relations \eqref{biggwarelations} of $S$ imply that $f$ factors through the preprojective algebra.  On the other hand, there is a homomorphism $g:S \to \Pi^{\bm{\alpha}}(Q)$ defined by $g(e_i) = e_i, g(h) = \sum_{i \in I}{a_ia^*_i}, g(a_i) = a_i$, and $g(a^*_i) = a^*_i$.  The reader should check that $g$ respects the defining relations \eqref{biggwarelations}.  Clearly $f$ and $g$ are mutually inverse.  
\end{example}

Let $\Pi_{ij}$ denote the $(\Pi_i,\Pi_j)$ bimodule $e_i \Pi e_j$.  Notice that $\Pi_{ij} \tensor_{\Pi_j} \Pi_{ji} = e_i \Pi e_j \Pi e_i \subset \Pi_i$.  It is straightforward to compute that 
\begin{align*}
\alpha_{ij} & = a_i a_{i+1} \dotsm a_{j-1} a_{j-1}^* \dotsm a^*_i = r_i \cdot \sg_i^{-1}(r_{i+1}) \cdot (\sg_i^{-1}\circ \sg_{i+1}^{-1})(r_{i+2}) \dotsm (\sg_i^{-1} \circ \dotsm \circ \sg_{j-2}^{-1})(r_{j-1}) \\
\beta_{ij} & = a_{i-1}^* a_{i-2}^* \dotsm a_j^* a_j \dotsm a_{i-1} = \sg_{i-1}(r_{i-1}) \cdot (\sg_{i-1} \circ \sg_{i-2})(r_{i-2}) \dotsm (\sg_{i-1} \circ \dotsm \circ \sg_j)(r_j)
\end{align*}
and these elements belong to the ideal $e_i \Pi e_j \Pi e_i \subset \Pi_i$.  Furthermore, with the notation of \eqref{vi}, we see that $v_i = \alpha_{ij} \theta_i^{-1}(\beta_{ij})$.  Set $\theta_{ij} = \sg_{i-1}  \circ \dotsm \circ \sg_{j+1} \circ \sg_j$ and observe that $\theta_{ij} \circ \theta_{ji} = \theta_i$ and 
\begin{align*}
 \theta_i(\alpha_{ij}) & =  \theta_{ij}( \beta_{ji}), & \beta_{ij} & = \theta_{ij}( \alpha_{ji}), \\
 \theta_j(\alpha_{ji}) & = \theta_{ji}( \beta_{ij}), & \beta_{ji} & = \theta_{ji}( \alpha_{ij}).
\end{align*}

\begin{lemma}\label{morfactor}
Let $(R,\theta,v)$ be GWA data.  Suppose that there is a factorization $v = uw$ such that the pairs $u,w$ and $u,\theta(w)$ are relatively prime.  Then $T(R,\theta,v)$ and $T(R,\theta, \theta(w)u)$ are (graded) Morita equivalent.
\end{lemma}
\begin{proof}
Let $Q$ be the oriented cycle of length 2.  Set $\sg_1 = \theta$ and $\sg_2 = \id$.  Let $r_1 = u$ and $r_2 = \theta(w)$ and $\Pi = \Pi(R,\sg,\br)$.  Then $T(R,\theta,v) \cong \Pi_1$ and $T(R,\theta, \theta(w)u) \cong \Pi_2$.  Using the notation above, we have
\begin{align*}
\alpha_{12} & = u & \beta_{12} & = \theta(w) \\
\alpha_{21} & = \theta(w) & \beta_{21} &= \theta(u) 
\end{align*}
By hypothesis the pairs $\alpha_{12},\beta_{12}$ and $\alpha_{21},\beta_{21}$ are coprime and therefore $\Pi_{ij} \tensor \Pi_{ji} = \Pi_i$ and $\Pi_{ji} \tensor \Pi_{ij} = \Pi_j$.  This means that the functors $\Pi_{ij} \tensor -$ and $\Pi_{ji} \tensor -$ induce inverse equivalences.
\end{proof}

Lemma \ref{morfactor} is a generalization of Theorem 2.3 and Lemma 2.4 in \cite{H}, where Hodges treats only the classical case.  In classical case we can get a more precise version of \ref{mornec}.  First we need to define an equivalence relation on polynomials in $\C[h]$.  Let $q:\C \to \C/\Z$ be the quotient map.  

\begin{defn}[Root type]  Let $\Z^f[h] \subset \Z[h]$ be the subset of totally factorizable integer polynomials $\Z^f[h] = \{ u(h) = \prod_i (h-n_i) : n_i \in \Z \}$.  Let $u_1 < u_2 < \dotsb < u_n$ and $v_1 < v_2 < \dotsb < v_m$ and put $u(h) = \prod_{i=1}^n{ (h-u_i)^{d_i} }$ and $v(h) = \prod_{i=1}^m{ (h-v_i)^{e_i}}$.  We say that $u$ and $v$ have the same type $u \sim v$ if $n = m$ and $d_i = e_i$ for all $i$.  This defines an equivalence relation on $\Z^f[h]$.  Now let $u,v$ be polynomials in $\C[h]$ and write $v(h) = \prod v_i(h-a_i), u(h) = \prod u_j(h-b_j)$ with $v_i,u_j \in \Z^f[h]$ such that the collections $\{a_i\}$ and $\{b_j\}$ are distinct modulo $\Z$.  We say that $u$ and $v$ have the same \emph{type}, $u \sim v$ if there is a bijection $i \mapsto j(i)$ such that
\begin{itemize}
	\item $q(a_i) = q(b_{j(i)})$ and
	\item $u_i \sim v_{j(i)}$.
\end{itemize}
Loosely speaking, two polynomials have the same root type if they have the same classes of roots modulo $\Z$ and if in each class of roots modulo $\Z$ considered under the natural ordering, the multiplicities occur in the same order.
\end{defn}

In order to prove Theorem \ref{classicalcase} (below) we need to know more about Artinian modules over classical GWAs.  Let $A$ be a classical GWA, with polynomial $v \in \C[h]$.  The connected components of $\spec \C[h]/(v)$ are just the roots of $v$.  Since ($\ast$) is satisifed, we have projective generators $P_\nu$ for $\Op$ indexed by the roots of $v$ and their simple quotients $S^\nu$.  We also have the small Verma modules $V^\nu = A/A(\tm, (h-\nu))$.  By \ref{simp}, $V^\nu$ has a submodule for each integer 	$k \geq 0$ such that $\nu + k$ is a root of $v$.  

\begin{defn}
Say that a graded module $M$ is $\nu$-small if $M$ has exactly one filtration $M = F^0 M \supset F^1 M \supset \dotsb F^{n-1}M \supset F^n M = 0$ such that $F^i M/F^{i+1} M \cong V^\nu$.  We set $\ell(M) = n$, the length of the unique filtration $F^\bt$ with $V^\nu$ quotients.
\end{defn}

\begin{lemma}
Define $M_\nu := P_\nu/A P_{-1}$.  Then $M_\nu$ is $\nu$-small and $\ell(M_\nu) = \mult(\nu,v)$.  If $M$ is $\nu$-small then any map $P_\nu \to M$ factors through a map $M_\nu \to M$.
\end{lemma}
\begin{proof}
Write $P_\nu = A(N)_\nu = A/A(\tm^N,(h-\nu)^f)$ where $f = \mult(\nu, \prod_{j=0}^{N-1}{v(h-j)})$.  Then $M_\nu$ is the quotient of $P_\nu$ by the submodule of $P_\nu$ generated by $\tm$.  The degree zero part is generated by $\tp^n \tm^n = \prod_{j=0}^{n-1}{v(h-j)}$ for $1 \leq n < N$.  So if we set $e = \mult(\nu,v)$, then $(M_\nu)_i \cong \C[h]/(h-\nu)^e$ as a right $\C[h]$ module for $i \geq 0$ and is zero otherwise.  Let $F^\bt$ be the filtration $F^i M_\nu = M_\nu (h-\nu)^i$, $0 \leq i \leq e$.  Clearly, $F^i M_\nu / F^{i+1} M_\nu = V^\nu$.  So if $M_\nu$ is $\nu$-small then $\ell(M_\nu) = e = \mult(\nu,v)$.

Let $G^i$, $0 \leq i \leq n$ be a filtration such that $G^iM_\nu / G^{i+1} M_\nu \cong V^\nu$.  Suppose that $G^i M_\nu = F^i M_\nu$ so that $G^i M_\nu = M_\nu (h-\nu)^i$ and $\dim_\C(G^i M_\nu / G^{i+1} M_\nu )_0 = 1$.  It follows that $G^{i+1} M_\nu$ contains $(h-\nu)^{i+1}$.  But then $G^{i+1} M_\nu$ contains $F^{i+1} M_\nu$ and since $V^\nu$ is not a nontrivial subquotient of itself, $G^{i+1}M_\nu = F^{i+1}M_\nu$.  Since $G^0 M_\nu = F^0 M_\nu = M_\nu$ we conclude that $G =F$ and that $M_\nu$ is $\nu$-small of length $\mult(\nu,v)$.

Suppose $M$ is $\nu$-small and consider a map $g:P_\nu \to M$.  Since $M$ is $\nu$-small, $M_i= 0$ for $i < 0$.  Therefore $P_{-1}$ is contained in the kernel of $g$ so $g$ desends to a map $M_\nu \to M$.
\end{proof}

\begin{lemma} \label{multtest}   Let $\nu$ be a root of $v$.  If $M$ is a $\nu$-small module then $\ell(M)$ is equal to the multiplicity of $S^\nu$ as a composition factor of $M$ and $\ell(M) \leq \mult(\nu,v)$.
\end{lemma}
\begin{proof}
Suppose that $M$ is a $\nu$-small module.  Let $G^i$, $0 \leq i \leq N$ be the unique filtration from the definition of $\nu$-small.  For each $i$ we have a surjection $G^i \onto V^\nu$.  Since $P_\nu$ is projective, the map $P_\nu \onto V^\nu$ lifts to a map $P_\nu \to G^i$.  Adding all these maps together we get a surjection $P_\nu^{\oplus N} \onto M$.  By the previous lemma this map factors through a map $M_\nu^{\oplus N} \onto M$.  Let $\tilde{F}^i$, $0 \leq i \leq \mult(\nu,v)$ be the filtration on $M_\nu^{\oplus N}$ induced by the filtration $F^i$ on $M_\nu$.  We have $\tilde{F}^i/\tilde{F}^{i+1} \cong (V^\nu)^{\oplus N}$.  Let $F^i$ also denote the image in $M$ of $\tilde{F}^i$.  For each $i$ we have $\tilde{F}^i/\tilde{F}^{i+1} \cong (V^\nu)^{\oplus N} \onto F^i/F^{i+1}$.  Let $K_i$ be the kernel of this map.  The simple subquotients of $V^\nu$ are naturally ordered $S_1,\dotsc,S_k$ such that if $j \leq l$ then $\mult(S_j, K_i) \leq \mult(S_l, K_i)$.  Now if $K_i$ contains $S_k$ as a subquotient, it must contain a direct summand.  Since the simple subquotients of $M$ all have the same multiplicity, $\ell(M)$, we see that $F^i/F^{i+1} \cong (V^\nu)^{\oplus N_i}$.  Of course if $N_i > 1$ for any $i$ then $M$ has infinitely many filtrations with successive quotients isomorphic to $V^\nu$.  Therefore $N_i \leq 1$ for all $i$, so $e \geq N$.
\end{proof} 

\begin{lemma} \label{charsmallvermas}
$V^\nu$ is the only module $M$ in $\Op$ such that i) $S^\nu$ is a quotient, ii) $\delta(M)$ is not bounded above, and iii) whenever $M'$ and $M''$ are submodules of $M$, either $M' \subset M''$ or $M'' \subset M'$.
\end{lemma}
\begin{proof}
Let $M$ be a module satisfying conditions i)-iii).  The surjection $P_\nu \onto S^\nu$ lifts to a map $P_\nu \to M$.  This map must be surjective by iii).  Since the multiplicity of $S^\nu$ as a subquotient is 1, the map $P_\nu$ factors through a map $V^\nu \onto M$.  Finally, no proper quotient of $V^\nu$ satisfies condition ii), so $V^\nu \cong M$.
\end{proof}

\begin{thm}[Classical case.] \label{classicalcase}
$T(v_1)$ and $T(v_2)$ are strongly graded Morita equivalent if and only if for some $b$, $v_1(h+b)$ and $v_2(h)$ have the same type.
\end{thm}

\begin{proof}
(Only if.)  Let $T_j = T(v_j)$ and assume that $F:T_1-\grMod \to T_2-\grMod$ is a strongly graded Morita equivalence.  Theorem \ref{mornec} uses $F$ to construct a $\tau$ equivariant automorphism $\psi$ such that $\nu$ is an integer translate of a root of $v_1$ if and only if $\psi(\nu)$ is an integer translate of a root of $v_2$.  Let $\Psi:T_1 \to T'_1 = T(\psi(v_1))$ be the isomorphism constructed just before Theorem \ref{mornec}.  We can view any graded $T'_1$ module $M$ as a graded $T_1$ module via $\Psi$ and this operation gives a strongly graded Morita equivalence $\Psi_*:T'_1 -\grMod \to T_1-\grMod$.  Now, $F \circ \Psi_*$ is a strongly graded Morita equivalence, but the equivariant automorphism associated to it is simply the identity.  Now, since $\psi$ commutes with integer translation, it must be a translation itself.  Therefore $\psi(v_1)(h) = v_1(h+b)$ for some $b$.  Thus, we are reduced to the case when $v_1$ and $v_2$ have the same classes of roots modulo $\Z$ and $\supp(F(M)) = \supp(M)$.

Let $\nu^i_w$ be the smallest root in each $\Z$ equivalence class of roots $w \in \C/\Z$ of $v_i$.  According to Theorem \ref{gr}, we can form categories $\D^i_w := \Op_w(T(v_i))$ for $w \in \C/\Z$ (where $\D^i_w = 0$ if $w$ is not a class of roots of$v_i$ modulo $\Z$) and we decompose $\Op(T(v_i))= \bigoplus_{n \in \Z, w \in \C/\Z}{\D^i_w[n]}$.  We will show that there is an $n_w \in \Z$ such that $F$ restricts to an equivalence between $\D^1_w$ and $\D^2_w[n_w]$.  Indeed, $\D^1_w$ is the thick subcategory generated by the indecomposable projectives $P_{w,k}$.  Since $P_{w,k}$ is indecomposable, so is $F(P_{w,k})$ and $\supp(F(P_{w,k})) = \supp(P_{w,k})$  so that $F(P_{w,k}) \in \D^2_w[n_w]$ for some $n_w$.  For each $k,k' \in \chi_w$ one of $\Homgr(P_{w,k},P_{w,k'})$ or $\Homgr(P_{w,k'},P_{w,k})$ is nonzero and it follows that $F(\D^1_w) \subset \D^2_w[n_w]$.  Parallel considerations for an inverse equivalence to $F$ imply that $F$ indeed restricts to an equivalence between $\D^1_w$ and $\D^2_w[n_w]$.

According to \ref{extcomp}, two nonisomorphic simple modules $S,T$ are adjacent if and only if $\Ext^1(S,T) \neq 0$.  Therefore nonisomorphic simple modules $S,T \in \D^1_w$ are adjacent if and only if $F(S)$ and $F(T)$ are adjacent.  Now, the simple modules in $\D^1_w$ are in bijection with the roots of $v_1$ congruent to $w$ modulo $\Z$.  Let $f$ be a bijection between the roots of $v_1$ and $v_2$ such that $F(S^\nu) = S^{f(\nu)}$.  There is exactly one isomorphism class of simple module $S$ in $\D^1_w$ such that $\delta(S)$ is unbounded above.  As in the proof of Theorem \ref{mornec}, $F(S)$ is also simple and $\delta(F(S))$ is unbounded above.  Therefore $f$ must identify the largest root of $v_1$ the equivalence class $w$ with the largest root of $v_2$ in $w$.  Since $S^\nu$ and $S^{\nu'}$ are adjacent if and only if there is no root $\eta$ of $v_i$ in the same equivalence class as $\nu,\nu'$ such that $\nu < \eta < \nu'$ or $\nu' < \eta < \nu$.  This means that $v_1$ and $v_2$ have the same type if $\mult(\nu,v_1) = \mult(f(\nu),v_2)$.

Observe that $V^{\nu_w + k}[-k]$ is in $\D^1_w$ for each $k$ such that $\nu_w + k$ is a root of $v_1$.  By \ref{multtest}, $\mult(\nu_w -k, v_1)$ is the maximum multiplicity of $S^{\nu_w + k}[-k]$ in any module $M[-k]$ such that $M$ is $\nu_w+k$ small.    Now, $F(V^\nu)$ has $S^{f(\nu)}$ as a quotient, $\delta(F(V^\nu))$ is unbounded above, and it satisfies condition iii) of Lemma \ref{charsmallvermas}.  Therefore Lemma \ref{charsmallvermas} implies that $F(V^\nu) \cong V^{f(\nu)}$.  So $\nu_w+k$ small modules go to $f(\nu_w)+k$ small modules and the multiplicity of $S^{\nu_w+k}[-k]$ in $M$ is the same as the multiplicity of $S^{f(\nu)+k}[-k]$ in $F(M)$.  We conclude that $\mult(f(\nu_w)+k,v_2) = \mult(\nu,v_1)$.

(If.)  Assume that for some $b$, $v_1(h)$ and $v_2(h+b)$ have the same type.  By the remarks at the beginning of the section, $T(v_2)$ and $T(v_2(h+b))$ are isomorphic as graded rings.  So we can assume that $v_1$ and $v_2$ have the same type.  First suppose that $\nu$ is a root of $v_1$ and $\eta$ is a root of $v_2$ with $\eta - \nu \in \Z_{>0}$ and such that there is no root of $v_2$ on the $\Z$ chain between $\nu$ and $\eta-1$.  Write $v_2(h) = w(h) (h-\eta)^e$ where $w(\eta) \neq 0$.  Then for each $0 \leq j \leq \eta - \nu$ the pairs $w(h),(h-\eta+j)^e$ and $w(h), (h-\eta+j+1)^e$ are relatively prime.  Hence the algebras $T(w(h)(h-\eta+j)^e)$ and $T(w(h)(h-\eta+j+1)^e)$ are graded Morita equivalent by \ref{morfactor}.  Now, replace $v_2(h)$ by $v_2(h+N)$ where $N$ is a large integer such that $v_2(h+N)$ and $v_1$ have no common roots.  We will ``move'' the roots of $v_2$ to the roots of $v_1$.  Using the previous argument we can move the smallest root of $v_2$ to the smallest root of $v_1$ in the same $\Z$ and then the next smallest and so on.  
\end{proof}

Let $R$ and $\sg$ be given.  Suppose that there is an automorphism $\psi \in \Aut(R)$ such that $\sg \psi \sg= \psi$.  Then there is an isomorphism $\Psi:T(R,\sg,v) \to T(R,\sg, \psi(\sg(v))$ which extends $\psi$ and satisfies $\Psi(\tp) = \tm$ and $\Psi(\tm) =\tp$.  Evidently this isomorphism is anti-graded in the sense that $\deg(x) + \deg(\Psi(x)) = 0$.  For a $\Z$ graded ring $T$ let $`T$ be the graded ring satisfying $`T_n = T_{-n}$.  So $\Psi$ defines an isomorphism between $T(R,\sg,v)$ and $`T(R,\sg,\psi(\sg(v)))$.  If we consider strongly anti-graded equivalences, i.e. strongly graded Morita equivalences between $T(R,\sg,v)$ and $`T(S,\theta,u)$ then we can still prove a version of \ref{mornec} where the equivariant isomorphism is replaced by an anti-equivariant isomorphism.

According to \cite{BJ} there are automorphisms of GWAs that are not graded or anti graded and therefore there are Morita equivalences that are not strongly graded or strongly anti-graded.  In these cases, however, the ordinary Morita equivalences can be replaced by strongly graded or strongly anti-graded ones.

\begin{question}
If two classical GWAs are Morita equivalent, are they then also strongly graded or strongly anti-graded Morita equivalent?
\end{question}

\bibliographystyle{plain}
\bibliography{gwas}

\def\cprime{$'$}
\begin{thebibliography}{1}

\bibitem{Bav2}
V.~Bavula.
\newblock Generalized {W}eyl algebras and their representations.
\newblock {\em Algebra i Analiz}, 4(1):75--97, 1992.

\bibitem{Bav1}
V.~Bavula.
\newblock Generalized {W}eyl algebras, kernel and tensor-simple algebras, their
  simple modules.
\newblock In {\em Representations of algebras ({O}ttawa, {ON}, 1992)},
  volume~14 of {\em CMS Conf. Proc.}, pages 83--107. Amer. Math. Soc.,
  Providence, RI, 1993.

\bibitem{BJ}
V.~Bavula and D.~A. Jordan.
\newblock Isomorphism problems and groups of automorphisms for generalized
  {W}eyl algebras.
\newblock {\em Trans. Amer. Math. Soc.}, 353(2):769--794 (electronic), 2001.

\bibitem{BGG}
I.~N. Bern{\v{s}}te{\u\i}n, I.~M. Gel{\cprime}fand, and S.~I. Gel{\cprime}fand.
\newblock A certain category of {${\mathfrak g}$}-modules.
\newblock {\em Funkcional. Anal. i Prilo\v zen.}, 10(2):1--8, 1976.

\bibitem{DGO}
Yuri~A. Drozd, Boris~L. Guzner, and Sergei~A. Ovsienko.
\newblock Weight modules over generalized {W}eyl algebras.
\newblock {\em J. Algebra}, 184(2):491--504, 1996.

\bibitem{H}
T.~J. Hodges.
\newblock Noncommutative deformations of type-{$A$} {K}leinian singularities.
\newblock {\em J. Algebra}, 161(2):271--290, 1993.

\bibitem{K}
A.~Khare.
\newblock Axiomatic framework for the {BGG} category {O}.
\newblock {\tt arXiv.org:math/0811.2080v1 [math.RT]}, 2008.

\bibitem{RS}
L.~Richard and A.~Solotar.
\newblock On {M}orita equivalence for simple generalized {W}eyl algebras.
\newblock {\tt arXiv.org:math/0805.3933 [math.KT]} (to appear in Algebras and
  Representation Theory), 2008.

\end{thebibliography}

\end{document}